\documentclass[amscd, amssymb,11pt]{article}

\usepackage[usenames]{color}
\definecolor{red}{rgb}{1.0,0.0,0.0}

\definecolor{blu}{rgb}{0.0,0.0,1.0}

\definecolor{gre}{rgb}{0.03,0.50,0.03}

\normalbaselines
\usepackage{amsmath}
\usepackage{amsthm}
\usepackage{amsfonts}
\usepackage{amssymb}
\usepackage{bbm}

\newtheorem{lemma}{Lemma}[section]

\newtheorem{theorem}{Theorem}[section]
\newtheorem{corollary}{Corollary}[section]
\newtheorem{remark}{Remark}[section]
\newtheorem{definition}{Definition}[section]
\newtheorem{example}{Example}[section]

\let\Section=\section
\def\section{\setcounter{equation}{0}\Section}
\sf

\def\Swiech
{{\accent"13S}wie{\hbox{\kern -0.21em\lower
0.79ex\hbox{$\textfont1=\scriptfont1
\lhook$}}}ch}
\def\SWIECH
{{\accent"13S}WIE{\hbox{\kern -0.26em\lower
0.77ex\hbox{$\textfont1=\scriptfont1
\lhook$}}}CH}

\begin{document}

\title{{\bf Interior regularity for nonlocal fully nonlinear equations with Dini continuous terms
}}

\author{
    \textsc{Chenchen Mou}\\
    \textit{School of Mathematics, Georgia Institute of Technology}\\
\textit{
Atlanta, GA 30332, U.S.A.}\\
 \textit{E-mail: cmou3@math.gatech.edu}   
  }
\date{}

\maketitle

\begin{abstract}
This paper is concerned with interior regularity of viscosity solutions of non-translation invariant nonlocal fully nonlinear equations with Dini continuous terms. We obtain $C^{\sigma}$ regularity estimates for the nonlocal equations by perturbative methods and a version of a recursive Evans-Krylov theorem.
\end{abstract}

\vspace{.2cm}

\vspace{.2cm}
\section{Introduction}
In this paper, we investigate interior regularity of viscosity solutions of nonlocal equations of the type
\begin{equation}\label{eq1.1}
\inf_{a\in\mathcal{A}}\Big\{\int_{\mathbb R^n}\big[u(x+y)-u(x)-\mathbbm{1}_{B_1(0)}(y)Du(x)\cdot y\big]K_a(x,y)dy\Big\}=f(x),\quad \text{in $B_1(0)$},
\end{equation}
where $\mathcal{A}$ is an index set, $\mathbbm{1}_{B_1(0)}$ denotes the indicator function of the unit ball $B_1(0)$ and $K_a(x,y)$ is a positive kernel. The kernels $K_a(x,y)$ are symmetric, i.e., for any $x,y\in\mathbb R^n$
\begin{equation}\label{sym}
K_a(x,y)=K_a(x,-y),
\end{equation}
and satisfy the uniform ellipticity assumption, i.e., for any $x\in\mathbb R^n$ and $y\in\mathbb R^n\setminus\{0\}$
\begin{equation}\label{ell}
\frac{(2-\sigma)\lambda}{|y|^{n+\sigma}}\leq K_a(x,y)\leq \frac{(2-\sigma)\Lambda}{|y|^{n+\sigma}},
\end{equation}
where $0<\lambda\leq \Lambda$. The symmetry assumption is essential for the regularity theory for $(\ref{eq1.1})$, see \cite{LS}. Under the symmetry assumption, $(\ref{eq1.1})$ can be rewritten as
\begin{equation*}
\inf_{a\in\mathcal{A}}\Big\{\int_{\mathbb R^n}\delta u(x,y)K_a(x,y)dy\Big\}=f(x),\quad \text{in $B_1(0)$},
\end{equation*}
where $\delta u(x,y)=u(x+y)+u(x-y)-2u(x)$. We furthermore assume that the kernels $K_a$ satisfy, for any $x\in\mathbb R^n$, any $y\in\mathbb R^n\setminus\{0\}$ and $i=1,2$
\begin{equation}\label{smo}
|D_y^iK(x,y)|\leq \frac{\Lambda(2-\sigma)}{|y|^{n+\sigma+i}}.
\end{equation}

We will obtain $C^{\sigma}$ regularity estimates for $(\ref{eq1.1})$ with Dini continuous data in two steps. We first generalize the recursive Evans-Krylov theorem for translation invariant nonlocal fully nonlinear equations from the case of H\"{o}lder continuous data, see \cite{TJ}, to the Dini continuous case. We then use the perturbative methods to obtain $C^{\sigma}$ regularity estimates for $(\ref{eq1.1})$.

 In recent years, regularity theory of viscosity solutions for integro-differential equations has been studied by many authors under uniform ellipticity assumption $(\ref{ell})$. It was initiated by a series of papers \cite{LL1,LL2,LL3} of L. A. Caffarelli and L. Silvestre, where $C^{\alpha}$ regularity, $C^{1+\alpha}$ regularity and Evans-Krylov theorem for nonlocal fully nonlinear elliptic equations were established. Later, H. Chang Lara and G. Davila studied these regularity results for nonlocal fully nonlinear parabolic equations, see in \cite{CD1,CD2,CD3}. In \cite{Kri}, D. Kriventsov used pertubative methods to prove $C^{1+\alpha}$ regularity estimates for nonlocal fully nonlinear elliptic equations with rough kernels. Then, in \cite{Se1}, J. Serra's results extended the results of \cite{Kri} to parabolic equations by a Liouville theorem and a blow up and compactness procedure. More recently, T. Jin and J. Xiong studied $C^{\sigma+\alpha}$ regularity in the $x$ variable for viscosity solutions for linear parabolic integro-differential equations.  In \cite{TJ}, T. Jin and J. Xiong proved $C^{\sigma+\alpha}$ regularity estimates for non-translation invariant nonlocal fully nonlinear elliptic equations using a recursive Evans-Krylov theorem and perturbative methods. At the same time, J. Serra refined and improved the method of \cite{Se1} to obtain $C^{\sigma+\alpha}$ regularity estimates for nonlocal equations with rough kernels, see \cite{Se}. The reader can also consult \cite{BCCI,BCI} for regularity results for a class of second order integro-differential equations with a different uniform ellipticity assumption. It allows nondegeneracy of the nonlocal terms, or nondegeneracy of nonlocal terms in some directions and nondegeneracy of second order terms in the complementary directions. We also refer the reader to \cite{L,LX,LC,N} for the $C^{\alpha}$ regularity, $C^{1+\alpha}$ regularity and Evans-Krylov theorem for classical fully nonlinear PDEs.

In Section 3, we establish a recursive Evans-Krylov theorem for translation invariant nonlocal fully nonlinear equations in the Dini continuous case. The sequence of equations we consider is, for $j=0,1,\cdots,m$
\begin{equation}\label{eqq3.1}
\inf_{a\in\mathcal{A}}\Big\{\int_{\mathbb R^n}\sum_{l=0}^{j}\rho^{-(j-l)\sigma}w^{-1}(\rho^j)w(\rho^l)\delta v_l(\rho^{j-l}x,\rho^{j-l}y)K_a^j(y)dy+w^{-1}({\rho^{j}})b_a\Big\}=0,\quad \text{in $B_5(0)$},
\end{equation}
where $w(t)$ is a Dini modulus of continuity, $K_a^j(x):=\rho^{j(n+\sigma)}K_a(\rho^jx)$ and $\rho\in(0,1)$. We prove that, for any $l=0,1,\cdots,m$, $\|v_l\|_{C^{\sigma+\bar\beta}(B_1(0))}\leq C$ where $0<\bar\beta<1$ and $C>0$ are two constants independent with $\rho$ and $m$. Recursive Evans-Krylov theorem was first studied by T. Jin and J. Xiong in \cite{TJ}. They used it to obtain the uniform regularity estimates for the approximators at each scale. Instead of using polynomials as approximators, they used solutions for constant coefficient equations since polynomials grow too fast near infinity. We construct a slightly more general recursive Evans-Krylov theorem for our purpose. When $w(t)=t^{\alpha}$ for some $0<\alpha<1$, $(\ref{eqq3.1})$ falls into the case in \cite{TJ}.

Having the recursive Evans-Krylov theorem in the Dini continuous case, in Section 4 we derive the main result of this manuscript, i.e., $C^{\sigma}$ regularity estimates of viscosity solutions for $(\ref{eq1.1})$ with Dini continuous data. To our knowledge, the only available results in this direction are about $C^{\sigma}$ regularity estimates for weak solutions of translation invariant nonlocal equations with bounded data. In Proposition $5.2$ of \cite{CS}, the authors proved $C^{\sigma}$ regularity estimates for
\begin{equation}\label{fra}
u=(-\Delta)^{-\frac{\sigma}{2}}f=\frac{1}{\Gamma (s)}\int_{0}^{+\infty}e^{t\Delta}f(x)\frac{dt}{t^{1-\frac{\sigma}{2}}},\quad\text{in $\mathbb R^n$},
\end{equation}
if $\sigma\not= 1$. For $\sigma=1$, they obtained $\Lambda_{*}(\mathbb R^n)$ regularity estimates for $(\ref{fra})$, where $\Lambda_{*}(\mathbb R^n)$ is the Zygmund space consisting of all bounded functions $u$ on $\mathbb R^n$ such that
\begin{equation*}
[u]_{\Lambda_*(\mathbb R^n)}:=\sup_{x,y\in\mathbb R^n}\frac{|u(x+y)+u(x-y)-2u(x)|}{|y|}<+\infty,
\end{equation*}
with the norm $\|u\|_{\Lambda_*(\mathbb R^n)}:=\|u\|_{L^{\infty}(\mathbb R^n)}+[u]_{\Lambda_*(\mathbb R^n)}$. It can be easily deduced from Proposition $2.8$ of \cite{L1} that the corresponding regularity estimates for weak solutions of $(-\Delta)^{\frac{\sigma}{2}}u=f$ in $\Omega$ hold. We notice that $C^1(\bar\Omega)\subsetneqq\Lambda_*(\Omega)$. In Theorem 1.1(b) of \cite{XJ}, it was shown that $C^{\sigma}$ regularity estimates for weak solutions hold for  
\begin{equation}\label{lev}
Lu:=\int_{\mathbb S^{n-1}}\int_{-\infty}^{\infty}\delta u(x,\theta r)\frac{dr}{|r|^{1+\sigma}}d\mu(\theta)=f(x),\quad\text{in $B_1(0)$},
\end{equation}
with a weaker ellipticity assumption 
\begin{equation*}
0<\lambda\leq\inf_{\nu\in \mathbb S^{n-1}}\int_{\mathbb S^{n-1}}|\nu\cdot \theta|^{\sigma}d\mu(\theta)\quad\text{and}\quad \mu(\mathbb S^{n-1})\leq\Lambda<+\infty,
\end{equation*}
where $\sigma\not =1$. If $\sigma=1$, the authors derived $C^{\sigma-\epsilon}$ regularity estimates for $(\ref{lev})$, where $\epsilon$ can be any positive constant between $0$ and $\sigma$. It was claimed in \cite{XJ} that the methods there can be applied to obtain similar regularity estimates for non-translation invariant equations. In \cite{DK}, H. Dong and D. Kim studied Schauder estimates for a class of nonlocal linear equations with rough kernels in both H\"older and Dini continuous case. However, in the Dini continuous case, they considered the global problem on translation invariant equations, i.e., $Lu=f$ in $\mathbb R^n$ where $L$ is defined in $(\ref{eqq2.1})$. Our results are different from the above results since we are considering the regularity theory of viscosity solutions for non-translation invariant nonlocal fully nonlinear equations. Weak solutions is not equivalent to viscosity solutions in general unless uniqueness of viscosity solutions for such equations holds. However, uniqueness of viscosity solutions for non-translation invariant nonlocal equations is still an open question. Some recent progress has been made in \cite{CA}. Finally we refer the reader to \cite{Ko1,Ko2} for $C^2$ regularity estimates for viscosity solutions of classical fully nonlinear PDEs with Dini continuous terms.

\section{Preliminaries}
Throughout this paper, $\Omega$ is always assumed to be a bounded domain in $\mathbb R^n$. For any $x\in\Omega$, we will write $u\in C^{1,1}(x)$, if there are a vector $p\in\mathbb R^n$, a constant $M>0$ and a neighborhood $N_x$ of $x$ such that
\begin{equation*}
|u(y)-u(x)-p\cdot (y-x)|\leq M|y-x|^2,\quad \text{for any $y\in N_x$}.
\end{equation*}
We denote by $L^1(\mathbb R^n,\frac{1}{1+|y|^{n+\sigma}})$ the usual weighted space of functions $u$ such that
\begin{equation*}
\|u\|_{L^1(\mathbb R^n,\frac{1}{1+|y|^{n+\sigma}})}:=\int_{\mathbb R^n}\frac{|u(y)|}{1+|y|^{n+\sigma}}dy<+\infty.
\end{equation*}

We recall some definitions and notation about nonlocal uniformly elliptic operators, see \cite{LL1,LL2,LL3}.
\begin{definition}
A nonlocal operator $I$ is an operator that maps a function u to a function $I[x,u]$ such that 
\begin{itemize}
\item[{\rm 1.}] $I[x,u]$ is well defined if $u\in C^{1,1}(x)$ and $u\in L^1(\mathbb R^n,\frac{1}{1+|y|^{n+\sigma}})$.
\item[{\rm 2.}] If $u\in C^{1,1}(\Omega)\cap L^1(\mathbb R^n,\frac{1}{1+|y|^{n+\sigma}})$, then $I[x,u]$ is continuous in $\Omega$ as a function of $x$.
\end{itemize}
\end{definition}
We say that the nonlocal operator $I$ is uniformly elliptic with respect to a class $\mathcal{L}$ of linear nonlocal operators if
\begin{equation*}
M_{\mathcal{L}}^-(u-v)(x)\leq I[x,u]-I[x,v]\leq M_{\mathcal{L}}^+(u-v)(x),
\end{equation*}
where 
\begin{equation*}
M_{\mathcal{L}}^+u(x):=\sup_{L\in\mathcal{L}}Lu(x),
\end{equation*}
\begin{equation*}
M_{\mathcal{L}}^-u(x):=\inf_{L\in\mathcal{L}}Lu(x).
\end{equation*}

The norm $\|I\|$ of a nonlocal operator $I$ is defined in the following way.
\begin{definition}
\begin{eqnarray*}
\|I\|_{\Omega}&:=&\sup\Big\{\frac{|I[x,u]|}{1+M}: x\in\Omega,u\in C^{1,1}(x),\|u\|_{L^1(\mathbb R^n,\frac{1}{1+|y|^{n+\sigma}})}\leq M,\\
&&|u(x+z)-u(x)-Du(x)\cdot z|\leq M|z|^2,\,\,\text{for any $z\in B_1(0)$}\Big\}.
\end{eqnarray*}
\end{definition}

The following classes of linear nonlocal operators $\mathcal{L}_i(\lambda,\Lambda,\sigma)$, $i=0,1,2,$ were introduced in \cite{LL1,LL2,LL3}. Let $0<\lambda\leq\Lambda$ be fixed constants. A linear nonlocal operator $L\in \mathcal{L}_0(\lambda,\Lambda,\sigma)$ if
\begin{equation}\label{eqq2.1}
Lu:=\int_{\mathbb R^n}\delta u(x,y)K(y)dy,
\end{equation}
where the kernel $K$ is symmetric and satisfies $(\ref{ell})$.
The class $\mathcal{L}_1(\lambda,\Lambda,\sigma)$ is a subclass of $\mathcal{L}_0(\lambda,\Lambda,\sigma)$ with kernels $K$ satisfying $(\ref{smo})$ with $i=1$. The class $\mathcal{L}_2(\lambda,\Lambda,\sigma)$ is a subclass of $\mathcal{L}_1(\lambda,\Lambda,\sigma)$ with kernels $K$ satisfying $(\ref{smo})$ with $i=2$. We note here that, for $i=0,1,2$, we will also write $K(y)\in\mathcal{L}_i(\lambda,\Lambda,\sigma)$ if the corresponding nonlocal operator $L\in\mathcal{L}_i(\lambda,\Lambda,\sigma)$.

We first review some properties of $L$ defined in $(\ref{eqq2.1})$, see \cite{TJ}.

\begin{lemma}\label{le2.1}
Suppose that $u\in C^4(B_2(0))\cap L^{\infty}(\mathbb R^n)$ and $L\in\mathcal{L}_2(\lambda,\Lambda,\sigma)$. Then
\begin{equation*}
\|Lu\|_{C^{2}(B_1(0))}\leq C(\|u\|_{C^4(B_2(0))}+\|u\|_{L^\infty(\mathbb R^n)}),
\end{equation*}
where $L$ is defined in $(\ref{eqq2.1})$ and $C$ is a positive constant depending on $n$, $\sigma_0$ and $\Lambda$.
\end{lemma}

\begin{lemma}\label{le2.2}
Suppose that $u\in C^{\sigma+\alpha}(\mathbb R^n)$, $0\leq K(y)\leq(2-\sigma)\Lambda|y|^{-n-\sigma}$ and $K(y)=K(-y)$. Then
\begin{equation*}
\|Lu\|_{C^{\alpha}(\mathbb R^n)}\leq C\|u\|_{C^{\sigma+\alpha}(\mathbb R^n)},
\end{equation*}
where $L$ is defined in $(\ref{eqq2.1})$ and $C$ is a positive constant depending on $n$, $\alpha$, $\sigma_0$ and $\Lambda$.
\end{lemma}

\begin{lemma}\label{le2.3}
Suppose that $u\in C^{\sigma+\alpha}(B_2(0))\cap L^{\infty}(\mathbb R^n)$, $0\leq K(y)\leq (2-\sigma)\Lambda |y|^{-n-\sigma}$, $K(y)=K(-y)$ and $|DK(y)|\leq \Lambda |y|^{-n-\sigma-1}$. Then
\begin{equation*}
\|Lu\|_{C^{\alpha}(B_1(0))}\leq C(\|u\|_{C^{\sigma+\alpha}(B_2(0))}+\|u\|_{L^{\infty}(\mathbb R^n)}),
\end{equation*}
where $L$ is defined in $(\ref{eqq2.1})$ and $C$ is a positive constant depending on $n$, $\alpha$, $\sigma_0$ and $\Lambda$.
\end{lemma}

\begin{lemma}\label{le2.4}
Let $v\in C_c^{\sigma+\alpha}(B_{\frac{1}{2}}(0))$ be such that $\|v\|_{C^{\sigma+\alpha}(B_{\frac{1}{2}}(0))}\leq 1$, and $p(x)$ be the Taylor polynomial of $v$ at $x=0$ of degree $[\sigma+\alpha]$. For any $L\in\mathcal{L}_0(\lambda,\Lambda,\sigma)$, there exists $P\in C_c^{\infty}(B_{\frac{1}{2}}(0))$ such that $P(x)=p(x)$ in $B_{\frac{1}{4}}(0)$, $\|P\|_{C^{4}(B_{\frac{1}{2}}(0))}\leq C$ and 
\begin{equation*}
LP(0)=Lv(0),
\end{equation*}
where $C$ is a positive constant depending on $n$, $\lambda$, $\Lambda$, $\sigma_0$ and $\alpha$.
\end{lemma}

We borrow the following two approximation lemmas from \cite{TJ}.
\begin{lemma}\label{le2.5}{\rm\cite[Lemma A.1]{TJ}}
For some $\sigma\geq\sigma_0>0$, we consider nonlocal operators $I_0$, $I_1$ and $I_2$ uniformly elliptic with respect to $\mathcal{L}_0(\lambda,\Lambda,\sigma)$. Assume that $I_0$ is translation invariant and $I_0(0)=1$.

Given $M>0$, a modulus of continuity $w_1$ and $\epsilon>0$, there exist $\eta_1>0$ and $R>5$ such that if $u$, $v$, $I_0$, $I_1$ and $I_2$ satisfy
\begin{equation*}
I_0(v,x)=0,\quad I_1(u,x)\geq -\eta_1\quad \text{and}\quad I_2(u,x)\leq \eta_1\quad \text{in $B_4(0)$}
\end{equation*}
in the viscosity sense, and 
\begin{equation*}
\|I_1-I_0\|_{B_4(0)}\leq \eta_1,\quad \|I_2-I_0\|_{B_4(0)}\leq \eta_1,
\end{equation*}
\begin{equation*}
u=v\quad \text{in $\mathbb R^n\setminus B_4(0)$},
\end{equation*}
\begin{equation*}
\|u\|_{L^{\infty}(\mathbb R^n)}\leq M\quad \text{in $\mathbb R^n$},
\end{equation*}
and 
\begin{equation*}
|u(x)-u(y)|\leq w_1(|x-y|)\quad \text{for any $x\in B_R(0)\setminus B_4(0)$ and $y\in\mathbb R^n\setminus B_4(0)$},
\end{equation*}
then $|u-v|\leq \epsilon$ in $B_4(0)$.
\end{lemma}
and 
\begin{lemma}\label{le2.6}{\rm\cite[Lemma A.2]{TJ}}
For some $\sigma\geq\sigma_0>0$, we consider nonlocal operators $I_0$, $I_1$ and $I_2$ uniformly elliptic with respect to $\mathcal{L}_0(\lambda,\Lambda,\sigma)$. Assume that 
\begin{equation*}
I_0v(x):=\inf_{a\in\mathcal{A}}\Big\{\int_{\mathbb R^n}\delta v(x,y)K_a(y)dy+h_a(x)\Big\}\quad \text{in $B_4(0)$},
\end{equation*}
where each $K_a\in \mathcal{L}_2(\lambda,\Lambda,\sigma)$ and for some constant $\beta\in (0,1)$,
\begin{equation*}
[h_a]_{C^{\beta}(B_4(0))}\leq M_0\quad\text{and}\quad \inf_{a\in\mathcal{A}}h_a(x)=0,\quad \text{for any $x\in B_4(0)$}.
\end{equation*}
Given $M_0$, $M_1$, $M_2$, $M_3>0$, $R_0>5$, $0<\beta,\nu<1$ and $\epsilon>0$, there exists $\eta_2$ such that if $u$, $v$, $I_0$, $I_1$ and $I_2$ satisfy
\begin{equation*}
I_0(v,x)=0,\quad I_1(u,x)\geq-\eta_2\quad \text{and}\quad I_2(u,x)\leq \eta_2\quad \text{in $B_4(0)$},
\end{equation*}
in the viscosity sense and
\begin{equation*}
\|I_1-I_0\|_{B_4(0)}\leq\eta_2,\quad \|I_2-I_0\|_{B_4(0)}\leq\eta_2
\end{equation*}
\begin{equation*}
u=v\quad \text{in $\mathbb R^n\setminus B_4(0)$},
\end{equation*}
\begin{equation*}
u=0\quad \text{in $\mathbb R^n\setminus B_{R_0}(0)$},
\end{equation*}
\begin{equation*}
\|u\|_{L^\infty(\mathbb R^n)}\leq M_1,
\end{equation*}
\begin{equation*}
[u]_{C^{\nu}(B_{R_0-\tau}(0))}\leq M_2\tau^{-4},\quad \text{for any $0<\tau<1$},
\end{equation*}
\begin{equation*}
[v]_{C^{\sigma+\beta}(B_{4-\tau}(0))}\leq M_3\tau^{-4},\quad \text{for any $0<\tau<1$},
\end{equation*}
then $|u-v|\leq \epsilon$ in $B_4(0)$.
\end{lemma}

We now introduce a modification of Evans-Krylov theorem for concave translation invariant nonlocal fully nonlinear equations.
\begin{theorem}\label{th1.1}{\rm\cite[Theorem 2.1]{TJ}}
Assume that $K_a(y)\in \mathcal{L}_2(\lambda,\Lambda,\sigma)$ with $2>\sigma\geq\sigma_0>1$ and $b_a$ is a constant for any $a\in\mathcal{A}$. If $u$ is a bounded viscosity solution of 
\begin{equation*}
\inf_{a\in\mathcal{A}}\Big\{\int_{\mathbb R^n}\delta u(x,y)K_a(y)dy+b_a\Big\}=0,\quad \text{in $B_1(0)$},
\end{equation*}
then $u\in C^{\sigma+\bar\alpha}(B_\frac{1}{2}(0))$ with
\begin{equation*}
\|u\|_{C^{\sigma+\bar\alpha}(B_\frac{1}{2}(0))}\leq C(\|u\|_{L^{\infty}(\mathbb R^n)}+|\inf_ab_a|),
\end{equation*}
where $\bar\alpha$ and $C$ are positive constants depending on $n$, $\sigma_0$, $\lambda$ and $\Lambda$.
\end{theorem}
In the rest of this paper, $\bar\alpha$ will always be the constant from Theorem \ref{th1.1}. We recall the definition of Dini modulus of continuity.
\begin{definition}
We say that $w(t)$ is a Dini modulus of continuity, if it satisfies
\begin{equation}\label{eq2.1}
\int_0^{t_0}\frac{w(r)}{r}dr<+\infty,\quad \text{for some $t_0>0$}.
\end{equation}
\end{definition}
We will make some additional assumption on our Dini modulus of continuity $w(t)$. Let $\bar\beta>0$ and $0<\sigma<2$.
\begin{itemize}
\item[{\rm $(H1)_{\bar\beta}$}]
There exists some $0<\beta<\bar\beta$ such that
\begin{equation}\label{H1}
\lim_{\mu\to0^+}\sup_{i\in\mathbb{N}}\frac{\mu^{\beta}w(\mu^i)}{w(\mu^{i+1})}=0.
\end{equation}
\item[{\rm $(H1)_{\bar\beta,\sigma}$}]
There exists some $0<\beta<\min\{2-\sigma,\bar\beta\}$ such that $(\ref{H1})$ holds.
\item[{\rm $(H2)_{\bar\beta,\sigma}$}]
Let $w(t)$ be a Dini modulus of continuity satisfying $(H1)_{\bar\beta,\sigma}$. There exists another Dini modulus of continuity $\tilde w(t)$ satisfying $(H1)_{\bar\beta,\sigma}$ such that, for any small $0<s\leq1$ and $0\leq t\leq 1$ we have
\begin{equation*}
w(st)\leq \eta(s) \tilde w(t),
\end{equation*}
where $\eta(s)$ is a positive function of $s$ such that $\lim_{s\to 0^+}\eta(s)=0$.
\end{itemize}
\begin{remark}
For any $\bar\beta>0$ and $0<\sigma<2$, we define 
\begin{equation*}
\mathcal{S}_{\bar\beta,\sigma}:=\{\text{Dini modulus of continuity satifying $(H2)_{\bar\beta,\sigma}$}\}.
\end{equation*}
It is obvious that $w(t)=t^{\alpha}\in \mathcal{S}_{\bar\beta,\sigma}$ for any $0<\alpha<\min\{\bar\beta,2-\sigma\}$ and $\cap_{\bar\beta>0,0<\sigma<2}\mathcal{S}_{\bar\beta,\sigma}$ does not contain any modulus of $w(t)=t^{\alpha}$.  
\end{remark}
\begin{lemma}
$\cap_{\bar\beta>0,0<\sigma<2}S_{\bar\beta,\sigma}\not=\emptyset$.
\end{lemma}
\begin{proof}
We claim that $w(t)=(\ln\frac{1}{t})^{\kappa-1}\in\cap_{\bar\beta>0,0<\sigma<2}\mathcal{S}_{\bar\beta,\sigma}$ for any $\kappa<0$. For any fixed $\bar\beta>0$ and $0<\sigma<2$, it is easy to verify that $w(t)$ is a Dini modulus of continuity satisfying $(H1)_{\bar\beta,\sigma}$. Now let us prove that $w(t)$ satisfies $(H2)_{\bar\beta,\sigma}$. For any $0<s<1$, we have
\begin{equation*}
w(st)=(\ln\frac{1}{st})^{\kappa-1}=\frac{(\ln\frac{1}{st})^{\kappa-1}}{(\ln\frac{1}{t})^{\frac{\kappa}{2}-1}}(\ln\frac{1}{t})^{\frac{\kappa}{2}-1}.
\end{equation*}
We notice that ($\ln\frac{1}{t})^{\frac{\kappa}{2}-1}$ is also a Dini modulus of continuity satisfying $(H1)_{\bar\beta,\sigma}$. For any $\epsilon>0$, there exists a sufficiently small constant $\delta_0>0$ depending only on $\epsilon$ such that
\begin{equation*}
\frac{(\ln\frac{1}{st})^{\kappa-1}}{(\ln\frac{1}{t})^{\frac{\kappa}{2}-1}}=\frac{(\ln\frac{1}{s}+\ln\frac{1}{t})^{\kappa-1}}{(\ln\frac{1}{t})^{\frac{\kappa}{2}-1}}<\epsilon,\quad\text{if $t< \delta_0$}.
\end{equation*}
Then there exists a sufficiently small constant $\delta_1>0$ depending only on $\epsilon$ such that 
\begin{equation*}
\frac{(\ln\frac{1}{st})^{\kappa-1}}{(\ln\frac{1}{t})^{\frac{\kappa}{2}-1}}<\epsilon,\quad\text{if $\delta_0\leq t<1$ and $0<s<\delta_1$}.
\end{equation*}
\end{proof}

\section{A recursive Evans-Krylov theorem}

The following theorem is a version of the recursive Evans-Krylov theorem we will use to prove $C^{\sigma}$ interior regularity.
\begin{theorem}\label{th3.1}
Assume that $2>\sigma\geq \sigma_0>0$, $b_a$ is a constant and $K_a(y)\in\mathcal{L}_2(\lambda,\Lambda,\sigma)$ for any $a\in\mathcal{A}$. Assume that $w$ is a modulus of continuity which satisfies $(H1)_{\bar\beta}$ where $\bar\beta$ depends on $n$, $\sigma_0$, $\lambda$, $\Lambda$. For each $m\in\mathbb N\cup\{0\}$, let $\{v_l\}_{l=0}^m$ be a sequence of functions satisfying $(\ref{eqq3.1})$
in the viscosity sense for any $j=0,1,\cdots,m$, where $K_a^j(x):=\rho^{j(n+\sigma)}K_a(\rho^jx)$ and $\rho\in (0,1)$. Suppose that $\|v_l\|_{L^{\infty}(\mathbb R^n)}\leq 1$ for any $l=0,1,\cdots,m$ and $|\inf_{a\in\mathcal{A}}b_a|\leq 1$. Then, there exist a sufficiently large constant $C>0$ and a sufficiently small constant $\rho_0>0$, both of which depend on $n$, $\sigma_0$, $\lambda$, $\Lambda$ and $w$, such that $v_l\in C^{\sigma+\bar\beta}(B_1(0))$ and, if $\rho\leq\rho_0$, we have
\begin{equation}\label{eq3.1}
\|v_l\|_{C^{\sigma+\bar\beta}(B_1(0))}\leq C,\quad \text{for any $l=0,1,\cdots,m$}.
\end{equation}
\end{theorem}
\begin{remark}\label{re3.1}
If $\sigma_0> 1$, then Theorem $\ref{th3.1}$ holds for $\bar\beta=\bar\alpha$.
\end{remark}
\noindent{\em Proof of Theorem \ref{th3.1}.} We will give the proof of Theorem \ref{th3.1} in the case $\sigma_0>1$. For the case $0<\sigma_0\leq 1$ the proof is similar. We adapt the approach from \cite{TJ}.

Let $M$ be a sufficiently large constant to be fixed later. By normalization, we can assume that 
\begin{equation*}
\|v_l\|_{L^{\infty}(\mathbb R^n)}\leq\frac{1}{M}\quad\text{and}\quad |\inf_{a\in\mathcal{A}}b_a|\leq\frac{1}{M},\quad\text{for any $l=0,1,\cdots,m$.}
\end{equation*}
Then we need to prove that $(\ref{eq3.1})$ holds for $C=1$. 

We will prove Theorem $\ref{th3.1}$ by induction on $m$. For the case of $m=0$, $(\ref{eq3.1})$ holds for $\bar\beta=\bar\alpha$ by Theorem $\ref{th1.1}$. Now we assume that Theorem \ref{th3.1} is true up to $m=i$ for any positive integer $i$. We want to show that the theorem is also true for $m=i+1$.
Define 
\begin{equation*}
R(x)=\sum_{l=0}^i\rho^{-(i-l)\sigma}w^{-1}(\rho^i)w(\rho^l)v_l(\rho^{i-l}x),
\end{equation*}
and, for any function $v$
\begin{equation*}
v_{\rho}^l(x)=\rho^{-\sigma}\frac{w(\rho^l)}{w(\rho^{l+1})}v(\rho x).
\end{equation*} 
By $(\ref{eqq3.1})$, we have
\begin{equation*}
\inf_{a\in\mathcal{A}}\{L_a^{i+1}R_{\rho}^i(x)+w^{-1}(\rho^{i+1})b_a\}=0,\quad\text{in $B_{\frac{5}{\rho}}(0)$},
\end{equation*}
where $L_a^{i+1}$ is the linear operator with kernel $K_a^{i+1}\in \mathcal{L}_2(\lambda,\Lambda,\sigma)$. Hence, there exists $\bar a\in\mathcal{A}$ such that 
\begin{equation}\label{eq3.3}
0\leq L_{\bar a}^{i+1}R_{\rho}^i(0)+w^{-1}(\rho^{i+1})b_{\bar a}<\rho^{\bar \alpha-\alpha},
\end{equation}
where $\alpha$ is given by $(H1)_{\bar\alpha}$. Let $\eta_0=1$ in $B_{\frac{1}{4}}(0)$ and $\eta_0\in C_c^{\infty}(B_{\frac{1}{2}}(0))$ be a fixed cut-off function. Let
\begin{equation*}
v_l=v_l\eta_0+v_l(1-\eta_0)=:v_l^{1}+v_{l}^{2},
\end{equation*}
and $p_l(x)$ be the Talyor polynomial of $v_l^{1}(x)$ at $x=0$ of degree $[\sigma+\bar\alpha]$. By Lemma \ref{le2.4}, there exists $P_l\in C_c^{\infty}(B_{\frac{1}{2}}(0))$ such that $P_l(x)=p_l(x)$ in $B_{\frac{1}{4}}(0)$ and $\|P_l\|_{C^{4}(B_{\frac{1}{2}}(0))}\leq C$ and 
\begin{equation}\label{eq3.4}
L_{\bar a}^{l}P_l(0)=L_{\bar a}^{l}v_l^{1}(0).
\end{equation}
Let 
\begin{equation*}
v_l=(v_l^{1}-P_l)+(v_l^{2}+P_l)=:V_l^{1}+V_l^{2}.
\end{equation*}
Thus, we have
\begin{equation*}
\|V_l^{1}\|_{L^{\infty}(\mathbb R^n)}+\|V_l^{2}\|_{L^{\infty}(\mathbb R^n)}\leq C,\quad V_l^{1}(0)=0,
\end{equation*}
\begin{equation}\label{eq3.5}
V_l^{1}\in C_c^{\sigma+\bar\alpha}(B_{\frac{1}{2}}(0)),\quad \|V_l^{1}\|_{C^{\sigma+\bar\alpha}(\mathbb R^n)}+\|V_l^{2}\|_{C^{\sigma+\bar\alpha}(B_1(0))}\leq C,
\end{equation}
\begin{equation*}
V_l^{1}=v_l-p_l\,\,\text{in $B_{\frac{1}{4}}(0)$}, V_l^{2}=p_l\,\,\text{in $B_{\frac{1}{4}}(0)$}, \|V_l^{1}(x)\|\leq C|x|^{\sigma+\bar\alpha}\,\,\text{in $\mathbb R^n$}.
\end{equation*}
Decompose $R(x)$ as
\begin{equation*}
R(x)=R^{(1)}(x)+R^{(2)}(x),
\end{equation*}
where 
\begin{equation*}
R^{(1)}(x)=\sum_{l=0}^i\rho^{-(i-l)\sigma}w^{-1}(\rho^i)w(\rho^l)V_l^{1}(\rho^{i-l}x),
\end{equation*}
and
\begin{equation*}
R^{(2)}(x)=\sum_{l=0}^i \rho^{-(i-l)\sigma}w^{-1}(\rho^i)w(\rho^l)V_l^{2}(\rho^{i-l}x).
\end{equation*}
Then, we have that, for each $a\in\mathcal{A}$，
\begin{eqnarray}\label{eq3.6}
L_a^{i+1}R_{\rho}^{(1)i}(x)&=&\sum_{l=0}^i\int_{\mathbb R^n}\rho^{-(i+1-l)\sigma}w^{-1}(\rho^{i+1})w(\rho^l)\delta V_l^{1}(\rho^{i+1-l}x,\rho^{i+1-l}y)K_a^{i+1}(y)dy\nonumber\\
&=&\sum_{l=0}^i\int_{\mathbb R^n}w^{-1}(\rho^{i+1})w(\rho^l)\delta V_l^{1}(\rho^{i+1-l}x,y)K_a^{l}(y)dy\nonumber\\
&=&\sum_{l=0}^i\frac{w(\rho^l)}{w(\rho^{i+1})}(L_a^{l}V_l^{1})(\rho^{i+1-l}x)
\end{eqnarray}
and
\begin{equation}\label{eqq3.7}
L_a^{i+1}R_\rho^{(2)i}(x)=\sum_{l=0}^i\frac{w(\rho^l)}{w(\rho^{i+1})}(L_a^{l}V_l^{2})(\rho^{i+1-l}x).
\end{equation}
It follows from $(\ref{eq3.3})$ and $(\ref{eq3.4})$ that
\begin{equation}\label{eq3.7}
L_{\bar a}^{i+1}R_\rho^{(1)i}(0)=0,
\end{equation}
\begin{equation}\label{eq3.8}
0\leq L_{\bar a}^{i+1}R_{\rho}^{(2)i}(0)+w^{-1}(\rho^{i+1})b_{\bar a}\leq \rho^{\bar\alpha-\alpha}.
\end{equation}
By $(H1)_{\bar\alpha}$, $(\ref{eq3.5})$, $(\ref{eq3.6})$, $(\ref{eq3.7})$ and Lemma \ref{le2.2}, we have, for any $x\in\mathbb R^n$
\begin{eqnarray}\label{eq3.9}
|L_{\bar a}^{i+1}R_{\rho}^{(1)i}(x)|&=&|L_{\bar a}^{i+1}R_{\rho}^{(1)i}(x)-L_{\bar a}^{i+1}R_{\rho}^{(1)i}(0)|\nonumber\\
&\leq&\sum_{l=0}^i\frac{w(\rho^l)}{w(\rho^{i+1})}|L_{a}^{l}V_l^{1}(\rho^{i+1-l}x)-L_a^{l}V_l^{1}(0)|\nonumber\\
&\leq&C|x|^{\bar\alpha}\sum_{l=0}^i\frac{w(\rho^l)}{w(\rho^{i+1})}\rho^{(i+1-l)\bar\alpha}\|V_l^{1}\|_{C^{\sigma+\bar\alpha}(\mathbb R^n)}\nonumber\\
&\leq&C|x|^{\bar\alpha}\sum_{l=0}^i \rho^{(i+1-l)(\bar\alpha-\alpha)}\nonumber\\
&\leq&C\rho^{\bar\alpha-\alpha}|x|^{\bar\alpha}.
\end{eqnarray}
Using $(H1)_{\bar\alpha}$, $(\ref{eq3.5})$, $(\ref{eqq3.7})$ and Lemma \ref{le2.3}, we have, for any $x\in B_5(0)$
\begin{eqnarray}\label{eq3.10}
|L_{\bar a}^{i+1}R_{\rho}^{(2)i}(x)-L_{\bar a}^{i+1}R_{\rho}^{(2)i}(0)|&\leq&\sum_{l=0}^i\frac{w(\rho^l)}{w(\rho^{i+1})}|L_a^{l}V_l^{2}(\rho^{i+1-l}x)-L_a^{l}V_l^{2}(0)|\nonumber\\
&\leq&C|x|^{\bar\alpha}\sum_{l=0}^i\frac{w(\rho^l)}{w(\rho^{i+1})}\rho^{(i+1-l)\bar\alpha}(\|V_l^{2}\|_{C^{\sigma+\bar\alpha}(B_1(0))}+\|V_l^{2}\|_{L^{\infty}(\mathbb R^n)})\nonumber\\
&\leq&C\rho^{\bar\alpha-\alpha}|x|^{\bar\alpha}.
\end{eqnarray}
Thus, by $(\ref{eq3.8})$ and $(\ref{eq3.10})$, we have
\begin{equation}\label{eq3.11}
|L_{\bar a}^{i+1}R_{\rho}^{(2)i}(x)+w^{-1}(\rho^{i+1})b_{\bar a}|\leq C\rho^{\bar\alpha-\alpha}(|x|^{\bar\alpha}+1),\quad \text{for any $x\in B_5(0)$}.
\end{equation}
We define
\begin{equation*}
\tilde v_{i+1}:=v_{i+1}+R_{\rho}^{(1)i}.
\end{equation*}
By ($\ref{eq3.5}$), we have
\begin{eqnarray}\label{eqq3.13}
|\tilde v_{i+1}(y)|&\leq&\|v_{i+1}\|_{L^{\infty}(\mathbb R^n)}+|R_\rho^{(1)i}(y)|\nonumber\\
&\leq&\frac{1}{M}+\sum_{l=0}^i\rho^{-(i+1-l)\sigma}w^{-1}(\rho^{i+1})w(\rho^l)V_l^{1}(\rho^{i+1-l}y)\nonumber\\
&\leq&\frac{1}{M}+\sum_{l=0}^i\rho^{-(i+1-l)(\sigma+\alpha)}|\rho^{i+1-l}y|^{\sigma+\bar\alpha}\nonumber\\
&\leq&\frac{1}{M}+\rho^{\bar\alpha-\alpha}|y|^{\sigma+\bar\alpha}.
\end{eqnarray}
By the definition of $\tilde v_{i+1}$, the following two equations are equivalent
\begin{equation}\label{eq3.12}
\inf_{a\in\mathcal{A}}\big\{L_a^{i+1}(v_{i+1}+R_{\rho}^i)(x)+w^{-1}(\rho^{i+1})b_a\big\}=0,\quad \text{in $B_5(0)$},
\end{equation}
and
\begin{equation}\label{eq3.13}
\inf_{a\in\mathcal{A}}\big\{L_a^{i+1}(\tilde v_{i+1}+R_{\rho}^{(2)i})(x)+w^{-1}(\rho^{i+1})b_a\big\}=0,\quad \text{in $B_5(0)$}.
\end{equation}
By $(\ref{eq3.9})$, $(\ref{eq3.11})$, $(\ref{eq3.12})$ and $(\ref{eq3.13})$, we have
\begin{equation}
L_{\bar a}^{i+1}v_{i+1}(x)\geq -C\rho^{\bar\alpha-\alpha},\quad\text{in $B_5(0)$},
\end{equation}
\begin{equation}
L_{\bar a}^{i+1}\tilde v_{i+1}(x)\geq -C\rho^{\bar\alpha-\alpha},\quad\text{in $B_5(0)$}.
\end{equation}

\begin{lemma}\label{le3.2}
Let $K$ be a symmetric kernel satisfying $0\leq K(y)\leq (2-\sigma)\Lambda|y|^{-n-\sigma}$. Then, for any smooth function $\tilde\eta$ such that
\begin{equation*}
0\leq\tilde\eta(x)\leq 1\,\,\text{in $\mathbb R^n$},\,\,\tilde\eta(x)=\tilde{\eta}(-x)\,\,\text{in $\mathbb R^n$},\,\,\tilde{\eta}(x)=0\,\,\text{in $\mathbb R^n\setminus B_{\frac{4}{5}}(0)$},\,\,\tilde\eta(x)=1\,\,\text{in $B_{\frac{3}{4}}(0)$},
\end{equation*}
we have
\begin{equation*}
M_{\mathcal{L}_2}^+\big(\tilde \eta(x)\int_{B_1(0)}\delta\tilde v_{i+1}(x,y)K(y)dy\big)\geq -C(\rho^{\bar\alpha-\alpha}+\frac{1}{M}),\quad\text{in $B_{\frac{3}{5}}(0)$}.
\end{equation*}
\end{lemma}
\begin{proof}
Define 
\begin{equation*}
\phi_k(y)=\mathbbm{1}_{B_1(0)\setminus B_{\frac{1}{k}}(0)}(y)K(y)
\end{equation*}
and
\begin{equation*}
T_kv(x)=\int_{\mathbb R^n}\delta v(x,y)\phi_k(y)dy,\quad\text{for any function $v$}.
\end{equation*}
By $(\ref{eq3.13})$, we have 
\begin{equation*}
L_a^{i+1}\tilde v_{i+1}(x)+L_a^{i+1}R_{\rho}^{(2)i}(x)+w^{-1}(\rho^{i+1})b_a\geq 0,\quad\text{for any $x\in B_3(0)$ and $a\in\mathcal{A}$}.
\end{equation*}
It follows that, for any $x\in B_{\frac{3}{2}}(0)$
\begin{eqnarray*}
0&\leq& \big(L_a^{i+1}\tilde v_{i+1}+L_a^{i+1}R_{\rho}^{(2)i}+w^{-1}(\rho^{i+1})b_a\big)*\phi_k(x)\\
&\leq& L_a^{i+1}\big(\tilde v_{i+1}*\phi_k\big)(x)+L_a^{i+1}R_{\rho}^{(2)i}*\phi_k(x)+w^{-1}(\rho^{i+1})b_a\|\phi_k\|_{L^1(\mathbb R^n)}.
\end{eqnarray*}
It also follows from $(\ref{eq3.13})$ that
\begin{equation*}
\inf_{a\in\mathcal{A}}\Big\{\|\phi_k\|_{L^1(\mathbb R^n)}\big(L_a^{i+1}\tilde v_{i+1}(x)+L_a^{i+1}R_{\rho}^{(2)i}(x)+w^{-1}(\rho^{i+1})b_a\big)\Big\}=0,\quad\text{for any $x\in B_3(0)$}.
\end{equation*}
Thus, for any $x\in B_{\frac{3}{2}}(0)$
\begin{equation*}
\sup_{a\in\mathcal{A}}L_a^{i+1}\big(\tilde v_{i+1}*\phi_k-\|\phi_k\|_{L^1(\mathbb R^n)}\tilde v_{i+1}\big)(x)+\sup_{a\in\mathcal{A}}\big\{L_a^{i+1}R_{\rho}^{(2)i}*\phi_k(x)-\|\phi_k\|_{L^1(\mathbb R^n)}L_a^{i+1}R_{\rho}^{(2)i}(x)\big\}\geq 0.
\end{equation*}
By (\ref{eq3.5}), (\ref{eqq3.7}) and Lemma \ref{le2.1}, we have, for any $x\in B_{\frac{3}{2}}(0)$ and $a\in\mathcal{A}$
\begin{eqnarray*}
&&2|L_a^{i+1}R_{\rho}^{(2)i}*\phi_k(x)-\|\phi_k\|_{L^{1}(\mathbb R^n)}L_a^{i+1}R_{\rho}^{(2)i}(x)|\\
&\leq&\Big|\int_{B_{1}(0)\setminus B_{\frac{1}{k}}(0)|}\delta\big(L_a^{i+1}R_{\rho}^{(2)i}\big)(x,y)K(y)dy\Big|\\
&\leq&\sum_{l=0}^{i}\frac{w(\rho^l)}{w(\rho^{i+1})}\int_{B_{1}(0)\setminus B_{\frac{1}{k}}(0)}\Big|\delta L_a^{l}V_l^{2}(\rho^{i+1-l}x,\rho^{i+1-l}y)\Big|K(y)dy\\
&\leq&\sum_{l=0}^i\rho^{(i+1-l)\sigma}\frac{w(\rho^l)}{w(\rho^{i+1})}\int_{B_{\rho^{i+1-l}}(0)\setminus B_{\frac{\rho^{i+1-l}}{k}}(0)}\Big|\delta L_a^{l}V_l^{2}(\rho^{i+1-l}x,y)\Big|K^{-(i+1-l)}(y)dy\\
&\leq&\sum_{l=0}^i\rho^{(i+1-l)\sigma}\frac{w(\rho^l)}{w(\rho^{i+1})}\int_{B_{\rho^{i+1-l}}(0)}\|L_a^{l}V_l^{2}\|_{C^{2}(B_{\frac{1}{8}}(0))}|y|^2 K^{-(i+1-l)}(y)dy\\
&\leq&C\sum_{l=0}^{i}\rho^{(i+1-l)\sigma}\frac{w(\rho^l)}{w(\rho^{i+1})}\big(\|V_l^{2}\|_{C^4(B_{\frac{1}{4}}(0))}+\|V_l^{2}\|_{L^{\infty}(\mathbb R^n)}\big)\int_{B_{\rho^{i+1-l}}(0)}\frac{(2-\sigma)\Lambda|y|^2}{|y|^{n+\sigma}}dy\\
&\leq&C\sum_{l=0}^i\rho^{(i+1-l)(2-\alpha)}\leq C\rho^{2-\alpha}.
\end{eqnarray*}
Therefore,
\begin{equation*}
M_{\mathcal{L}_2}^+\big(\tilde v_{i+1}*\phi_k-\|\phi_k\|_{L^1(\mathbb R^n)}\tilde v_{i+1}\big)(x)\geq -C\rho^{2-\alpha},\quad\text{in $B_{\frac{3}{2}}(0)$.}
\end{equation*}
Thus, we have
\begin{equation}\label{eq3.18}
M_{\mathcal{L}_2}^+(T_k\tilde v_{i+1})(x)\geq -C\rho^{2-\alpha},\quad \text{in $B_{\frac{3}{2}}(0)$}.
\end{equation}
Let $\bar L$ be any operator with kernel $\bar K\in \mathcal{L}_2(\lambda,\Lambda,\sigma)$. For any $x\in B_{\frac{3}{5}}(0)$, we have 
\begin{eqnarray}\label{eq3.19}
\bar L(\tilde \eta T_k \tilde v_{i+1})(x)&=&\int_{\mathbb R^n}\delta(T_k\tilde v_{i+1})(x,y)\bar K(y)dy-\int_{\mathbb R^n}\delta\big((1-\tilde\eta)T_k \tilde v_{i+1}\big)(x,y)\bar K(y)dy\nonumber\\
&=&\bar L\Big(T_k \tilde v_{i+1}\Big)(x)-2\int_{\mathbb R^n}\big(1-\tilde \eta(x-y)\big)T_k \tilde v_{i+1}(x-y)\bar K(y)dy.
\end{eqnarray}
We now estimate the second term in $(\ref{eq3.19})$. For any $x\in B_{\frac{3}{5}}(0)$
\begin{eqnarray}\label{eq3.20}
&&\Big|\int_{\mathbb R^n}T_k v_{i+1}(x-y)\big(1-\tilde\eta (x-y)\big)\bar K(y)dy\Big|=\Big|\int_{\mathbb R^n}v_{i+1}(x-y)T_k\big((1-\tilde\eta(x-\cdot))\bar K(\cdot)\big)(y)dy\Big|\nonumber\\
&\leq&\|v_{i+1}\|_{L^{\infty}(\mathbb R^n)}\int_{\mathbb R^n}\int_{B_1(0)}\Big|\big(1-\tilde\eta(x-y-z)\big)\bar K(y+z)+\big(1-\tilde\eta(x-y+z)\big)\bar K(y-z)\nonumber\\
&&-2\big(1-\tilde\eta(x-y)\big)\bar K(y)\Big|K(z)dzdy\nonumber\\
&\leq&C\|v_{i+1}\|_{L^{\infty}(\mathbb R^n)}\leq\frac{C}{M},
\end{eqnarray}
and, by $(\ref{eq3.5})$ and Lemma \ref{le2.2},
\begin{eqnarray}\label{eq3.21}
&&|T_kR_{\rho}^{(1)i}(x)|=\Big|\int_{B_1(0)\setminus B_{\frac{1}{k}}(0)}\delta R_\rho^{(1)i}(x,y)K(y)dy\Big|\nonumber\\
&=&\Big|\sum_{l=0}^i\int_{B_1(0)\setminus B_{\frac{1}{k}}(0)}\rho^{-(i+1-l)\sigma}w^{-1}(\rho^{i+1})w(\rho^l)\delta V_l^{1}(\rho^{i+1-l}x,\rho^{i+1-l}y)K(y)dy\Big|\nonumber\\
&=&\Big|\sum_{l=0}^i\int_{B_{\rho^{i+1-l}}(0)\setminus B_{\frac{\rho^{i+1-l}}{k}}(0)}w^{-1}(\rho^{i+1})w(\rho^l)\delta V_l^{1}(\rho^{i+1-l}x,y)K^{-(i+1-l)}(y)dy\Big|\nonumber\\
&\leq&\sum_{l=0}^i w^{-1}(\rho^{i+1})w(\rho^l)\Big|\int_{B_{\rho^{i+1-l}}(0)\setminus B_{\frac{\rho^{i+1-l}}{k}}(0)}\big(\delta V_l^{1}(\rho^{i+1-l}x,y)-\delta V_l^{1}(0,y)\big)K^{-(i+1-l)}(y)dy\Big|\nonumber\\
&&+\sum_{l=0}^iw^{-1}(\rho^{i+1})w(\rho^l)\Big|\int_{B_{\rho^{i+1-l}}(0)\setminus B_{\frac{\rho^{i+1-l}}{k}}(0)}\delta V_l^{1}(0,y)K^{-(i+1-l)}(y)dy\Big|\nonumber\\
&\leq&C\sum_{l=0}^iw^{-1}(\rho^{i+1})w(\rho^l)|\rho^{i+1-l}x|^{\bar\alpha}\nonumber\\
&&+C\sum_{l=0}^i w^{-1}(\rho^{i+1})w(\rho^l)\int_{B_{\rho^{i+1-l}}(0)}\frac{(2-\sigma)\Lambda|y|^{\sigma+\bar\alpha}}{|y|^{n+\sigma}}dy\nonumber\\
&\leq&C\rho^{\bar\alpha-\alpha}(1+|x|^{\bar\alpha}).
\end{eqnarray}
Since $\sigma>\bar\alpha$ holds, we have, for any $x\in B_{\frac{3}{5}}(0)$
\begin{eqnarray}\label{eq3.22}
&&\Big|\int_{\mathbb R^n}\big(1-\tilde\eta(x-y)\big)T_k R_\rho^{(1)i}(x-y)\bar K(y)dy\Big|\nonumber\\
&=&\Big|\int_{\mathbb R^n}(1-\tilde\eta(y))T_kR_\rho^{(1)i}(y)\bar K(x-y)dy\Big|\nonumber\\
&=&\Big|\int_{\mathbb R^n\setminus B_{\frac{3}{4}}(0)}\big(1-\tilde \eta(y)\big)T_kR_\rho^{(1)i}(y)\bar K(x-y)dy\Big|\nonumber\\
&\leq&C\rho^{\bar\alpha-\alpha}\int_{|y|>\frac{1}{64}}\frac{(2-\sigma)\Lambda}{|y|^{n+\sigma-\bar\alpha}}dy\leq C\rho^{\bar\alpha-\alpha}.
\end{eqnarray}
Taking the supremum of all $\bar K\in \mathcal L_2(\lambda,\Lambda,\sigma)$ in $(\ref{eq3.19})$ and using $(\ref{eq3.18})$, $(\ref{eq3.20})$ and $(\ref{eq3.22})$, we have, for any $x\in B_{\frac{3}{5}}(0)$
\begin{eqnarray*}
M_{\mathcal{L}_2}^+(\tilde \eta T_k \tilde v_{i+1})(x)&\geq& -C\rho^{2-\alpha}-\frac{C}{M}-C\rho^{\bar\alpha-\alpha}\\
&\geq&-C(\rho^{\bar\alpha-\alpha}+\frac{1}{M}).
\end{eqnarray*}
By Theorem $\ref{th1.1}$, we know that $\tilde v_{i+1}\in C^{\sigma+\bar\alpha}(B_4(0))$. Thus 
\begin{equation*}
\int_{B_1(0)\setminus B_{\frac{1}{k}}(0)}\delta\tilde v_{i+1}(x,y)K(y)dy\to\int_{B_1(0)}\delta\tilde v_{i+1}(x,y)K(y)dy,\quad\text{in $B_{\frac{3}{2}}(0)$ uniformly},
\end{equation*}
as $k\to+\infty$. It is obvious that 
\begin{equation*}
\tilde\eta(x)\int_{B_1(0)\setminus B_{\frac{1}{k}}(0)}\delta\tilde v_{i+1}(x,y)K(y)dy\to\tilde\eta(x)\int_{B_1(0)}\delta\tilde v_{i+1}(x,y)K(y)dy,\quad\text{in $L^{1}(\mathbb R^n,\frac{1}{1+|x|^{n+\sigma}})$}.
\end{equation*}
Thus, the result follows by Lemma 5 in \cite{LL2}.
\end{proof}
\begin{lemma}\label{le3.3}
There is a constant $C$ depending on $n,\sigma_0,\lambda,\Lambda$ such that, for any operator $L$ with a symmetric kernel $K$ satisfying $0\leq K(y)\leq (2-\sigma)\Lambda|y|^{n+\sigma}$ we have
\begin{equation*}
|Lv_{i+1}(x)|\leq C(\rho^{\bar\alpha-\alpha}+\frac{1}{M}),\quad\text{in $B_1(0)$}.
\end{equation*}
\begin{proof}
The proof follows from that of Lemma 2.9 and Lemma 2.10 in \cite{TJ}.
\end{proof}
\end{lemma}
\begin{lemma}\label{le3.4}
There is a constant $C$ depending on $n,\sigma_0,\lambda,\Lambda$ such that
\begin{equation}\label{eq3.23}
\max\{|M_{\mathcal{L}_0}^+v_{i+1}|,|M_{\mathcal{L}_0}^-v_{i+1}|\}\leq C(\rho^{\bar\alpha-\alpha}+\frac{1}{M}),\quad \text{in $B_1(0)$}.
\end{equation}
Moreover, we have 
\begin{equation}\label{eq3.24}
\|\nabla v_{i+1}\|_{L^{\infty}(B_{\frac{1}{2}}(0))}\leq C(\rho^{\bar\alpha-\alpha}+\frac{1}{M}),
\end{equation}
and
\begin{equation}\label{eq3.25}
\|\nabla \tilde v_{i+1}\|_{L^{\infty}(B_{\frac{1}{2}}(0))}\leq C(\rho^{\bar\alpha-\alpha}+\frac{1}{M}).
\end{equation}
\end{lemma}
\begin{proof}
($\ref{eq3.23}$) follows directly from Lemma \ref{le3.3}. To prove ($\ref{eq3.24}$), we first notice that $v_{i+1}$ satisfies
\begin{equation*}
\inf_{a\in\mathcal{A}}\{L_a^{i+1}(v_{i+1}+R_\rho^i)(x)+w^{-1}(\rho^{i+1})b_a\}=0,\quad \text{in $B_5(0)$}.
\end{equation*}
We define 
\begin{equation*}
I^0\cdot(x)=\inf_{a\in\mathcal{A}}\{L_a^{i+1}\cdot (x)+L_a^{i+1}R_\rho^i(0)+w^{-1}(\rho^{i+1})b_a\}.
\end{equation*}
By Theorem \ref{th1.1}, we know that $I^0$ has $C^{\sigma+\bar\alpha}$ estimates. By ($\ref{eq3.9}$) and ($\ref{eq3.10}$), we have that $v_{i+1}$ is a bounded function solves
\begin{equation*}
I^{0}v_{i+1}(x)\leq -\inf_{a\in\mathcal{A}}\{L_a^{i+1}R_\rho^i(x)-L_a^{i+1}R_\rho^i(0)\}\leq C\rho^{\bar\alpha-\alpha},\quad\text{in $B_1(0)$},
\end{equation*}
and
\begin{equation*}
I^0v_{i+1}(x)\geq-\sup_{a\in\mathcal{A}}\{L_a^{i+1}R_\rho^i(x)-L_a^{i+1}R_\rho^i(0)\}\geq -C\rho^{\bar\alpha-\alpha},\quad\text{in $B_1(0)$}.
\end{equation*}
It follows from Theorem 5.2 in \cite{LL2} that $v_{i+1}\in C^{1,\alpha_1}(B_{\frac{1}{2}}(0))$ for any $\alpha_1<\sigma_0-1$ and 
\begin{equation*}
\|v_{i+1}\|_{C^{1,\alpha_1}(B_{\frac{1}{2}}(0))}\leq C(\frac{1}{M}+\rho^{\bar\alpha-\alpha}).
\end{equation*}
By $(\ref{eq3.5})$, we have $|\nabla V_l^{1}(x)|\leq C|x|^{\sigma+\bar\alpha-1}$ in $B_{\frac{1}{2}}(0)$. Thus, for any $x\in B_{\frac{1}{2}}(0)$ we have
\begin{eqnarray*}
|\nabla R_\rho^{(1)i}(x)|&=&\Big|\nabla\sum_{l=0}^i\rho^{-(i+1-l)\sigma}w^{-1}(\rho^{i+1})w(\rho^l)V_l^{1}(\rho^{i+1-l}x)\Big|\\
&\leq&C\sum_{l=0}^i\rho^{-(i+1-l)(\sigma+\alpha-1)}\rho^{(i+1-l)(\sigma+\bar\alpha-1)}\\
&\leq&C\sum_{l=0}^i\rho^{(i+1-l)(\bar\alpha-\alpha)}\leq C\rho^{\bar\alpha-\alpha}.
\end{eqnarray*}
Thus, $(\ref{eq3.25})$ follows.
\end{proof}
\begin{lemma}\label{th3.2}
There is a constant $C$ depending on $n,\sigma_0,\lambda,\Lambda$ such that
\begin{equation*}
\int_{\mathbb R^n}|\delta v_{i+1}(x,y)|\frac{2-\sigma}{|y|^{n+\sigma}}dy\leq C(\rho^{\bar\alpha-\alpha}+\frac{1}{M})\quad \text{in $B_1(0)$}.
\end{equation*}
\end{lemma}
\begin{proof}
By Lemma \ref{le3.3} and \ref{le3.4}, it follows from the proof of Theorem 7.4 in \cite{LL3}.
\end{proof}

Let $\tilde\eta$ be the smooth function in Lemma \ref{le3.2}. For any symmetric measurable set $A$, we define
\begin{equation*}
w_A(x):=\tilde\eta(x)\int_{B_1(0)}\big(\delta\tilde v_{i+1}(x,y)-\delta\tilde v_{i+1}(0,y)\big)K_A(y)dy,
\end{equation*}
where
\begin{equation*}
K_A(y)=\frac{2-\sigma}{|y|^{n+\sigma}}\mathbbm{1}_A(y).
\end{equation*}
By Lemma \ref{le2.2}, we have for any $x\in B_1(0)$
\begin{eqnarray}\label{e3.27}
&&\Big|\int_{B_1(0)}\big(\delta R_\rho^{(1)i}(x,y)-\delta R_\rho^{(1)i}(0,y)\big)K_A(y)dy\Big|\nonumber\\
&=&\Big|\sum_{l=0}^i\rho^{-(i+1-l)\sigma}w^{-1}(\rho^{i+1})w(\rho^l)\int_{B_1(0)}\big(\delta V_l^{1}(\rho^{i+1-l}x,\rho^{i+1-l}y)-\delta V_l^{1}(0,\rho^{i+1-l}y)\big)K_A(y)dy\Big|\nonumber\\
&=&\Big|\sum_{l=0}^iw^{-1}(\rho^{i+1})w(\rho^l)\int_{B_1(0)}\big(\delta V_l^{1}(\rho^{i+1-l}x,y)-\delta V_l^{1}(0,y)\big)K_A^{l-1-i}(y)dy\Big|\nonumber\\
&\leq&\sum_{l=0}^i\rho^{-(i+1-l)\alpha}\|V_l^{1}\|_{C^{\sigma+\bar\alpha}(\mathbb R^n)}\rho^{(i+1-l)\bar\alpha}|x|^{\bar\alpha}\leq C\rho^{\bar\alpha-\alpha}|x|^{\bar\alpha}.
\end{eqnarray}
Using Lemma \ref{th3.2} and $(\ref{e3.27})$, we get
\begin{equation*}
|w_A|\leq C(\rho^{\bar\alpha-\alpha}+\frac{1}{M}),\quad\text{in $\mathbb R^n$}.
\end{equation*}
It follows from Lemma \ref{le3.3} and ($\ref{eq3.21}$) that
\begin{equation*}
\Big|\int_{B_1(0)}\delta \tilde v_{i+1}(0,y)K_A(y)dy\Big|\leq C(\rho^{\bar\alpha-\alpha}+\frac{1}{M}).
\end{equation*}
By Lemma \ref{le3.2}, we have
\begin{equation*}
M_{\mathcal{L}_2}^+w_A\geq -C(\rho^{\bar\alpha-\alpha}+\frac{1}{M}),\quad\text{in $B_{\frac{3}{5}}(0)$ uniformly in $A$}.
\end{equation*}
We define
\begin{equation*}
P(x):=\sup_A w_A(x)=\tilde\eta (x)\int_{B_1(0)}\big(\delta\tilde v_{i+1}(x,y)-\delta \tilde v_{i+1}(0,y)\big)^+\frac{2-\sigma}{|y|^{n+\sigma}}dy,
\end{equation*}
and
\begin{equation*}
N(x):=\sup_{A} -w_A(x)=\tilde \eta(x)\int_{B_1(0)}\big(\delta \tilde v_{i+1}(x,y)-\delta \tilde v_{i+1}(0,y)\big)^-\frac{2-\sigma}{|y|^{n+\sigma}}dy.
\end{equation*}
\begin{lemma}\label{le3.5}
For any $x\in B_{\frac{1}{4}}(0)$, we have
\begin{equation}\label{eqq3.28}
\frac{\lambda}{\Lambda}N(x)-C(\rho^{\bar\alpha-\alpha}+\frac{1}{M})|x|\leq P(x)\leq \frac{\lambda}{\Lambda}N(x)+C(\rho^{\bar\alpha-\alpha}+\frac{1}{M})|x|.
\end{equation}
\end{lemma}
\begin{proof}
For any $x\in B_{\frac{1}{2}}(0)$, we define $\tilde v_{i+1,x}(z):=\tilde v(x+z)$. By ($\ref{eq3.13}$), we have
\begin{equation*}
M_{\mathcal{L}_2}^+(\tilde v_{i+1,x}-\tilde v_{i+1})(0)\geq -\sup_{a\in\mathcal{A}}(L_a^{i+1}R_\rho^{(2)i}(x)-L_a^{i+1}R_\rho^{(2)i}(0))
\end{equation*}
and
\begin{equation*}
M_{\mathcal{L}_2}^-(\tilde v_{i+1,x}-\tilde v_{i+1})(0)\leq \sup_{a\in\mathcal{A}}(L_a^{i+1}R_\rho^{(2)i}(0)-L_a^{i+1}R_\rho^{(2)i}(x)).
\end{equation*}
By Lemma \ref{le2.1} and $(\ref{eq3.5})$,
\begin{eqnarray*}
|L_a^{i+1}R_\rho^{(2)i}(x)-L_a^{i+1}R_\rho^{(2)i}(0)|&=&\Big|\sum_{l=0}^i\frac{w(\rho^l)}{w(\rho^{i+1})}\big(L_a^{l}V_l^{2}(\rho^{i+1-l}x)-L_a^{l}V_l^{2}(0)\big)\Big|\\
&\leq&C\sum_{l=0}^i\rho^{-(i+1-l)\alpha}\big(\|V_l^{2}\|_{C^{4}(B_{\frac{1}{4}}(0))}+\|V_l^{2}\|_{L^{\infty}(\mathbb R^n)}\big)\rho^{i+1-l}|x|\\
&\leq&C\rho^{1-\alpha}|x|.
\end{eqnarray*}
Thus, we have
\begin{equation}\label{eq3.27}
M_{\mathcal{L}_2}^+(\tilde v_{i+1,x}-\tilde v_{i+1})(0)\geq -C\rho^{1-\alpha}|x|\,\,\text{and}\,\,M_{\mathcal{L}_2}^-(\tilde v_{i+1,x}-\tilde v_{i+1})(0)\leq C\rho^{1-\alpha}|x|.
\end{equation}
For any $L\in\mathcal{L}_2(\lambda,\Lambda,\sigma)$, we have
\begin{eqnarray}\label{eq3.28}
L(\tilde v_{i+1,x}-\tilde v_{i+1})(0)&=&\int_{\mathbb R^n}\big(\delta \tilde v_{i+1}(x,y)-\delta v_{i+1}(0,y)\big)K(y)dy\nonumber\\
&=&\int_{B_1(0)}(\delta \tilde v_{i+1}(x,y)-\delta \tilde v_{i+1}(0,y))K(y)dy\nonumber\\
&&+\int_{\mathbb R^n\setminus B_1(0)}(\delta \tilde v_{i+1}(x,y)-\delta \tilde v_{i+1}(0,y))K(y)dy.
\end{eqnarray}
By ($\ref{eqq3.13}$), ($\ref{eq3.25}$) and $L\in\mathcal{L}_2(\lambda,\Lambda,\sigma)$, we have, for any $x\in B_{\frac{1}{4}}(0)$
\begin{eqnarray}
&&\frac{1}{2}\int_{\mathbb R^n\setminus B_1(0)}\big(\delta \tilde v_{i+1}(x,y)-\delta \tilde v_{i+1}(0,y)\big)K(y)dy\nonumber\\
&=&\int_{\mathbb R^n}\tilde v_{i+1}(y)\big(K(y-x)\mathbbm{1}_{B_1^c(0)}(y-x)-K(y)\mathbbm{1}_{B_1^c(0)}(y)\big)dy\nonumber\\
&&-\big(\tilde v_{i+1}(x)-\tilde v_{i+1}(0)\big)\int_{\mathbb R^n\setminus B_1(0)}K(y)dy\nonumber\\
&\leq&\int_{\mathbb R^n\setminus B_{1}(0)}|\tilde v_{i+1}(y)||K(y-x)-K(y)|dy\nonumber\\
&&+\|\tilde v_{i+1}\|_{L^{\infty}(B_{1+|x|}(0))}\int_{B_{1+|x|}(0)\setminus B_{1-|x|}(0)}K(y)dy+C(\rho^{\bar\alpha-\alpha}+\frac{1}{M})|x|\nonumber\\
&&\leq C(\rho^{\bar\alpha-\alpha}+\frac{1}{M})|x|.
\end{eqnarray}
Therefore, we have
\begin{eqnarray}\label{eq3.31}
&&\int_{\mathbb R^n}(\delta \tilde v_{i+1}(x,y)-\delta \tilde v_{i+1}(0,y))K(y)dy\nonumber\\
&\leq&\int_{B_1(0)}(\delta \tilde v_{i+1}(x,y)-\delta \tilde v_{i+1}(0,y))K(y)dy+C(\rho^{\bar\alpha-\alpha}+\frac{1}{M})|x|.
\end{eqnarray}
By ($\ref{eq3.27}$) and ($\ref{eq3.31}$), we obtain
\begin{eqnarray*}
-C\rho^{1-\alpha}|x|&\leq& M_{\mathcal{L}_2}^+(\tilde v_{i+1,x}-\tilde v_{i+1})(0)\\
&\leq&\sup_{\frac{\lambda(2-\sigma)}{|y|^{n+\sigma}}\leq K\leq \frac{\Lambda(2-\sigma)}{|y|^{n+\sigma}}}\int_{B_1(0)}(\delta \tilde v_{i+1}(x,y)-\delta \tilde v_{i+1}(0,y))K(y)dy+C(\rho^{\bar\alpha-\alpha}+\frac{1}{M})|x|.
\end{eqnarray*}
Therefore, we have
\begin{equation*}
\Lambda P(x)-\lambda N(x)\geq -C(\rho^{\bar\alpha-\alpha}+\frac{1}{M})|x|.
\end{equation*}
The second inequality of ($\ref{eqq3.28}$) follows from $M_{\mathcal{L}_2}^-(\tilde v_{i+1,x}-\tilde v_{i+1})(0)\leq C\rho^{1-\alpha}|x|$.
\end{proof}

Now the proof of Theorem \ref{th3.1} follows from the proofs of Lemma 2.14 and Theorem 2.2 in \cite{TJ}.

\section{$C^{\sigma}$ regularity}

Before introducing the main theorem, we remind that, for any $\sigma\in (0,2)$, $[\sigma]$ denotes the largest integer which is less than or equal to $\sigma$.
\begin{theorem}\label{th4.1}
Assume that $2>\sigma\geq \sigma_0>0$ and $K_a(x,y)\in \mathcal{L}_2(\lambda,\Lambda,\sigma)$ for any $a\in\mathcal{A}$. Assume that $w(t)$ is a Dini modulus of continuity satisfying $(H2)_{\bar\beta,\sigma}$, where $\bar\beta$ is given in Theorem \ref{th3.1}. Assume that $f$ satisfies, for some $C_f>0$,
\begin{equation}\label{sa}
|f(x)-f(0)|\leq C_f w(|x|)\,\, \text{and}\,\, |f(x)|\leq C_f,\quad \text{in $B_1(0)$},
\end{equation}
and $K_a(x,y)$ satisfies, for any $0<r\leq 1$ and $a\in\mathcal{A}$
\begin{equation}\label{hs}
\int_{\mathbb R^n}|K_a(x,y)-K_a(0,y)|\min\{|y|^{\min\{2,\sigma+\bar\beta\}},r^{\min\{2,\sigma+\bar\beta\}}\}dy\leq \Lambda w(|x|)r^{\min\{2-\sigma,\bar\beta\}},\quad \text{in $B_1(0)$}.
\end{equation}
If $u$ is a bounded viscosity solution of $(\ref{eq1.1})$, then there exists a polynomial $p(x)$ of degree $[\sigma]$ such that
\begin{equation*}
|u(x)-p(x)|\leq C(\|u\|_{L^{\infty}(\mathbb R^n)}+C_f)|x|^{\sigma}\psi(|x|),\quad \text{in $B_\frac{1}{2}(0)$},
\end{equation*}
and
\begin{equation*}
|D^ip(0)|\leq C(\|u\|_{L^{\infty}(\mathbb R^n)}+C_f),\quad i=0,\cdots,[\sigma],
\end{equation*}
where $\psi(t):=w(t)+\int_{0}^t\frac{w(r)}{r}dr$ and $C$ is a constant depending on $\lambda,\Lambda,n,\sigma_0$, $\sigma$ and $w$.
\end{theorem}
\begin{proof}
By covering and rescaling arguments, we can assume ($\ref{eq1.1}$), ($\ref{sa}$) and ($\ref{hs}$) hold in $B_5(0)$. We will give the proof of Theorem \ref{th4.1} in the most complicated case $\sigma_0\geq1$. Without loss of generality, we can assume that $w(1)>1$.

We claim that we can find a sequence of functions $\{u_l\}_{l=0}^{l=+\infty}$ such that, for any $\rho\leq\rho_0$, $0<\kappa\leq\sigma+\bar\beta$ and $i=0,1,2,\cdots$, we have
\begin{equation}\label{eq4.7}
\inf_{a\in\mathcal{A}}\Big\{\int_{\mathbb R^n}\sum_{l=0}^{i}\delta u_l(x,y)K_a(0,y)dy\Big\}=f(0),\quad \text{in $B_{4\rho^i}$}(0),
\end{equation}
\begin{equation}\label{eq4.8}
(u-\sum_{l=0}^iu_l)(\rho^ix)=0,\quad \text{in $\mathbb R^n\setminus B_4(0)$},
\end{equation}
\begin{equation}\label{eq4.9}
\|u_i\|_{L^{\infty}(\mathbb R^n)}\leq \rho^{\sigma i}w(\rho^i),
\end{equation}
\begin{equation}\label{eq4.11}
\|u_i\|_{C^{\kappa}(B_{(4-\tau)\rho^i}(0))}\leq C_2\rho^{(\sigma-\kappa)i}w(\rho^i)\tau^{-\kappa},
\end{equation}
\begin{equation}\label{eq4.12}
\|u-\sum_{l=0}^iu_l\|_{L^{\infty}(\mathbb R^n)}\leq \rho^{\sigma(i+1)}w(\rho^{i+1}),
\end{equation}
\begin{equation}\label{eq4.13}
[u-\sum_{l=0}^iu_l]_{C^{\alpha_1}(B_{(4-3\tau)\rho^i}(0))}\leq 8C_1\rho^{(\sigma-\alpha_1)i}w(\rho^i)\tau^{-3},
\end{equation}
where $\rho_0$ is given by Theorem \ref{th3.1}, $\tau$ is an arbitrary constant in $(0,1]$, $\alpha_1$ and $C_1$ are positive constants depending on $n$, $\lambda$, $\Lambda$, $\sigma_0$, and $C_2$ is the constant in $(\ref{eq3.1})$. 

Suppose that we have $(\ref{eq4.7})$-$(\ref{eq4.13})$. Then, for any $\rho^{i+1}\leq |x|< \rho^i$
\begin{eqnarray*}
\Big|u(x)-\sum_{l=0}^{+\infty}u_l(0)-\sum_{l=0}^{+\infty}\nabla u_l(0)\cdot x\Big|&\leq&\Big|u(x)-\sum_{l=0}^{i}u_l(x)\Big|+\Big|\sum_{l=0}^{i}\big(u_l(x)-u_l(0)-\nabla u_l(0)\cdot x\big)\Big|\\
&&+\Big|\sum_{l=i+1}^{+\infty}u_l(0)\Big|+\Big|\sum_{l=i+1}^{+\infty}\nabla u_l(0)\cdot x\Big|\\
&\leq&\rho^{\sigma(i+1)}w(\rho^{i+1})+C|x|^{\min\{2,\sigma+\bar\beta\}}\sum_{l=0}^{i}\rho^{-\min\{2-\sigma,\bar\beta \}l}w(\rho^l)\\
&&+\sum_{l=i+1}^{+\infty}\rho^{\sigma l}w(\rho^{l})+C|x|\sum_{l=i+1}^{+\infty}\rho^{(\sigma-1)l}w(\rho^l).
\end{eqnarray*}
By $(H1)_{\bar\beta,\sigma}$, we have, for $\rho^{i+1}\leq |x|<\rho^i$
\begin{eqnarray*}
|x|^{\min\{2,\sigma+\bar\beta\}}\sum_{l=0}^{i}\rho^{-\min\{2-\sigma,\bar\beta \}l}w(\rho^l)
&\leq&\rho^{i\sigma}w(\rho^i)\sum_{l=0}^{i}\rho^{\min\{2-\sigma,\bar\beta\}(i-l)}\frac{w(\rho^l)}{w(\rho^i)}\\
&\leq&\rho^{i\sigma}w(\rho^i)\sum_{l=0}^{i}\rho^{\big(\min\{2-\sigma,\bar\beta\}-\beta\big)(i-l)}\\
&\leq&\rho^{i\sigma}w(\rho^i)\sum_{l=0}^{+\infty}\rho^{\big(\min\{2-\sigma,\bar\beta\}-\beta\big)l}\\
&\leq&C\rho^{i\sigma}w(\rho^i)\leq C\rho^{-\beta-\sigma}\rho^{(i+1)\sigma}w(\rho^{i+1})\frac{\rho^{\beta}w(\rho^i)}{w(\rho^{i+1})}\\
&\leq&C\rho^{(i+1)\sigma}w(\rho^{i+1}).
\end{eqnarray*}
We notice that $\min\{2,\sigma+\bar\beta\}-\min\{2-\sigma,\bar\beta\}=\sigma$. Thus, for $\rho^{i+1}\leq |x|<\rho^i$
\begin{eqnarray*}
&&\Big|u(x)-\sum_{l=0}^{+\infty}u_l(0)-\sum_{l=0}^{+\infty}\nabla u_l(0)\cdot x\Big|\\
&\leq&C\rho^{\sigma(i+1)}w(\rho^{i+1})+\big(\rho^{\sigma(i+1)}
+C\rho^i\rho^{(\sigma-1)(i+1)}\big)\sum_{l=i+1}^{+\infty}w(\rho^l)\\
&\leq&C\rho^{\sigma(i+1)}w(\rho^{i+1})+C\rho^{\sigma(i+1)}\sum_{l=i+1}^{+\infty}w(\rho^l)\\
&\leq&C\rho^{\sigma(i+1)}\psi(\rho^{i+1}),
\end{eqnarray*}
where $\psi(t)=w(t)+\int_{0}^t\frac{w(r)}{r}dr$.

We first prove the claim for $i=0$. Let $u_0$ be the viscosity solution of 
\begin{equation}\label{eq4.1}
\left\{\begin{array}{ll}
I_0u_0:=\inf_{a\in\mathcal{A}}\Big\{\int_{\mathbb R^n}\delta u_0(x,y)K_a(0,y)\Big\}-f(0)=0,\quad &\hbox{in $ B_4(0)$},\\
\qquad \qquad \qquad \qquad \qquad \qquad \qquad \qquad\quad\qquad\, u_0=u,
&\hbox{in $B_4^c(0)$}.\end{array}\right.
\end{equation}
Then, by Lemma 3.1 in \cite{TJ}, we have 
\begin{equation}
\|u_0\|_{L^{\infty}(\mathbb R^n)}\leq C(\|u\|_{L^{\infty}(\mathbb R^n)}+\|f\|_{L^{\infty}(B_5(0))}).
\end{equation}
By normalization, we can assume that  
\begin{equation*}
\|u_0\|_{L^{\infty}(\mathbb R^n)}\leq \frac{1}{2}\quad \text{and}\quad \|u\|_{L^{\infty}(\mathbb R^n)}+\|f\|_{L^{\infty}(B_5(0))}\leq \frac{1}{2}.
\end{equation*}
Using Theorem \ref{th3.1}, we have, for any $0<\kappa\leq \sigma+\bar\beta$
\begin{equation*}
\|u_0\|_{C^{\kappa}(B_{4-\tau}(0))}\leq C_2\tau^{-\kappa},
\end{equation*}
where $C_2$ is the constant in $(\ref{eq3.1})$. Since $u$ is a bounded viscosity solution of $(\ref{eq1.1})$, it follows from Theorem $12.1$ in \cite{LL1} that there exist constants $\alpha_1>0$ and $C_1>0$, depending only on $n$, $\lambda$, $\Lambda$, $\sigma_0$, such that, for any $0<\tau\leq 1$
\begin{equation}\label{eq4.14}
\|u\|_{C^{\alpha_1}(B_{4-\tau}(0))}\leq \frac{C_1}{2}\tau^{-\alpha_1}.
\end{equation}
Let $\epsilon:=\rho^{\sigma+\bar\beta}\leq \rho^{\sigma}w(\rho)$, $M=1$ and $w_1(r):=r^{\alpha_1}$. Then, for these $w_1$, $\epsilon$ and $M$, there exist $\eta_1>0$ and $R>5$ such that Lemma \ref{le2.5} holds. 
Without loss of generality, we can assume that, for any $0<r\leq 1$
\begin{equation*}
|f(x)-f(0)|\leq\gamma w(|x|),\quad \text{in $B_5(0)$},
\end{equation*}
\begin{equation*}
\int_{B_r(0)}|K_a(x,y)-K_a(0,y)||y|^{\min\{2,\sigma+\bar\beta\}}dy\leq\gamma w(|x|)r^{\min\{2-\sigma,\bar\beta\}},\quad \text{in $B_5(0)$},
\end{equation*}
\begin{equation*}
\int_{B_r^c(0)}|K_a(x,y)-K_a(0,y)|dy\leq\gamma w(|x|)r^{-\sigma},\quad \text{in $B_5(0)$},
\end{equation*}
\begin{equation}\label{eq4.18}
|u(x)-u(y)|\leq w_1(|x-y|),\quad \text{for any $x\in B_R(0)\setminus B_4(0)$ and $y\in \mathbb R^n\setminus B_4(0)$},
\end{equation}
where $\gamma$ is a sufficiently small constant we determine later. This can be achieved by scaling. For a sufficiently small $s>0$, if we let
\begin{equation*}
\tilde K_a(x,y)=s^{n+\sigma}K_a(sx,sy)\in\mathcal{L}_2(\lambda,\Lambda,\sigma),
\end{equation*}
\begin{equation*}
\tilde u(x)=u(sx),
\end{equation*}
\begin{equation*}
\tilde f(x)=s^{\sigma}f(sx),
\end{equation*}
then we see that
\begin{equation*}
\tilde I \tilde u(x):=\inf_{a\in\mathcal{A}} \int_{\mathbb R^n}\delta \tilde u(x,y) \tilde K_a(x,y)dy=\tilde f(x),\quad \text{in $B_5(0)$}.
\end{equation*}
It follows from $(H2)_{\bar\beta,\sigma}$ that, if we choose $s$ sufficiently small, then for any $x\in B_5(0)$
\begin{equation*}
|\tilde f(x)-\tilde f(0)|\leq C_f s^{\sigma}w(s|x|)\leq C_f s^{\sigma}\eta(s) \tilde w(|x|)\leq \gamma \tilde w(|x|),
\end{equation*}
\begin{eqnarray*}
&&\int_{B_r(0)}|\tilde K_a(x,y)-\tilde K_a(0,y)||y|^{\min\{2,\sigma+\bar\beta\}}dy\\
&=&s^{-\min\{2-\sigma,\bar\beta\}}\int_{B_{sr}(0)}|K_a(sx,y)-K_a(0,y)||y|^{\min\{2,\sigma+\bar\beta\}}dy\\
&\leq&\Lambda w(s|x|)r^{\min\{2-\sigma,\bar\beta\}}\leq \Lambda\eta(s)\tilde w(|x|)r^{\min\{2-\sigma,\bar\beta\}}\leq\gamma\tilde w(|x|)r^{\min\{2-\sigma,\bar\beta\}},
\end{eqnarray*}
and
\begin{eqnarray*}
\int_{B_r^c(0)}|\tilde K_a(x,y)-\tilde K_a(0,y)|dy&=&s^{\sigma}\int_{B_{sr}^c(0)}|K_a(sx,y)-K_a(0,y)|dy\\
&\leq&\Lambda w(s|x|) r^{-\sigma}\leq \Lambda \eta(s)\tilde w(|x|)r^{-\sigma}\leq \gamma\tilde w(|x|)r^{-\sigma},
\end{eqnarray*}
where $\tilde w(t)$ is another Dini modulus of continuity satisfying $(H1)_{\bar\beta,\sigma}$ and $\eta(s)$ is a positive function of $s$ such that $\lim_{s\to 0^+}\eta(s)=0$. Using $(\ref{eq4.14})$ with $\tau=1$, we have, if we let $s$ sufficiently small,
\begin{equation*}
\|\tilde u\|_{C^{\alpha_1}(B_{2R}(0))}\leq\|\tilde u\|_{C^{\alpha_1}(B_{\frac{3}{s}
}(0))}\leq  s^{\alpha_1}\frac{C_1}{2}\leq 1.
\end{equation*}
Since $R>5$ and $\|\tilde u\|_{C^{\alpha_1}(B_{2R}(0))}\leq 1$,
\begin{eqnarray*}
|\tilde u(x)-\tilde u(y)|\leq |x-y|^{\alpha_1}\quad \text{for any $x\in B_R(0)\setminus B_4(0)$ and $y\in B_{2R}(0)\setminus B_4(0)$},
\end{eqnarray*}
and
\begin{eqnarray*}
|\tilde u(x)-\tilde u(y)|\leq 1\leq |x-y|^{\alpha_1}\quad \text{for $x\in B_R(0)\setminus B_4(0)$ and $y\in B_{2R}^c(0)$}.
\end{eqnarray*}
Therefore, $(\ref{eq4.18})$ holds for $\tilde u$.

If $x\in B_4(0)$, $h\in C^{1,1}(x)$, $\|h\|_{L^{\infty}(\mathbb R^n)}\leq M$ and $|h(y)-h(x)-(y-x)\cdot\nabla h(x)|\leq \frac{M}{2}|x-y|^2$ for any $y\in B_1(x)$, we have
\begin{eqnarray}
\|I-&I_0&\|_{B_4(0)}\leq\frac{1}{M+1}\Big\{\int_{\mathbb R^n}|\delta h(x,y)||K_a(x,y)-K_a(0,y)|dy+f(x)-f(0)\Big\}\nonumber\\
&\leq&\frac{M}{M+1}\Big\{\int_{B_1(0)}|y|^2|K_a(x,y)-K_a(0,y)|\Big\}dy+4\int_{\mathbb R^n\setminus B_1(0)}|K_a(x,y)-K_a(0,y)|dy\Big\}\nonumber\\
&&+f(x)-f(0)\leq6\gamma w(|x|)\leq 6\gamma w(5).\label{eq4.6}
\end{eqnarray}
We will choose $\gamma<\min\{\frac{\eta_1}{6w(5)},\frac{1}{(C_2+4)w(4)}\}$. By Lemma \ref{le2.5}, we have 
\begin{equation*}
\|u-u_0\|_{L^{\infty}(B_4(0))}\leq\epsilon\leq\rho^{\sigma}w(\rho),
\end{equation*}
and thus
\begin{equation*}
\|u-u_0\|_{L^{\infty}(\mathbb R^n)}\leq\|u-u_0\|_{L^{\infty}(B_4(0))}\leq\epsilon\leq\rho^{\sigma}w(\rho).
\end{equation*}

Let $v(x)=u(x)-u_0(x)$. Since $u_0\in C_{\rm loc}^{\sigma+\bar\beta}(B_4(0))$, $v$ is a viscosity solution of 
\begin{eqnarray*}
I^{(0)}v(x):&=&\inf_{a\in\mathcal{A}}\Big\{\int_{\mathbb R^n}\delta v(x,y)K_a(x,y)+\delta u_0(x,y)K_a(x,y)dy\Big\}-f(0)\\
&=&f(x)-f(0)\quad \text{in $B_4(0)$}.
\end{eqnarray*} 
It is clear that $I^{(0)}$ is uniformly elliptic with respect to $\mathcal{L}_0(\lambda,\Lambda,\sigma)$. Since $\gamma<\frac{1}{(C_2+4)w(4)}$, we have for any $x\in B_{4-2\tau}(0)$
\begin{eqnarray*}
|I^{(0)}0(x)|&=&\Big|\inf_{a\in\mathcal{A}}\Big\{\int_{\mathbb R^n}\delta u_0(x,y)K_a(x,y)dy\Big\}-f(0)\Big|\\
&\leq&\sup_{a\in\mathcal{A}}\int_{\mathbb R^n}|\delta u_0(x,y)||K_a(x,y)-K_a(0,y)|dy\\
&\leq&\sup_{a\in\mathcal{A}}\Big\{\int_{B_\tau(0)}C_2\tau^{-\min\{2,\sigma+\bar\beta\}}|y|^{\min\{2,\sigma+\bar\beta\}}|K_a(x,y)-K_a(0,y)|dy\\
&&+4\int_{\mathbb R^n\setminus B_{\tau}(0)}|K_a(x,y)-K_a(0,y)|dy\Big\}\\
&\leq&\gamma C_2\tau^{-\min\{2,\sigma+\bar\beta\}}w(|x|)\tau^{\min\{2-\sigma,\bar\beta\}}+4\gamma w(|x|)\tau^{-\sigma}\\
&=&\gamma(C_2+4)w(|x|)\tau^{-\sigma}\leq \gamma(C_2+4)w(4)\tau^{-\sigma}\leq \tau^{-\sigma}.
\end{eqnarray*}
It follows from Theorem 12.1 in \cite{LL1} that
\begin{equation*}
\|v\|_{C^{\alpha_1}(B_{4-3\tau}(0))}\leq C_1\tau^{-\alpha_1}(\tau^{-\sigma}+w(4)\gamma+1)\leq 8C_1\tau^{-3},
\end{equation*}
and thus
\begin{equation*}
[u-u_0]_{C^{\alpha_1}(B_{4-3\tau}(0))}\leq 8C_1\tau^{-3}.
\end{equation*}

We then assume $(\ref{eq4.7})$-$(\ref{eq4.13})$ hold up to $i\geq 0$ and we will show that they hold for $i+1$ as well. Let
\begin{equation*}
U(x)=\rho^{-(i+1)\sigma}w^{-1}(\rho^{i+1})\big(u-\sum_{l=0}^iu_l\big)(\rho^{i+1}x),
\end{equation*}
\begin{equation*}
v_l(x)=\rho^{-l\sigma}w^{-1}(\rho^l)u_l(\rho^lx),
\end{equation*}
and 
\begin{equation*}
K_a^{i+1}(x,y)=\rho^{(n+\sigma)(i+1)}K_a(\rho^{i+1}x,\rho^{i+1}y).
\end{equation*}
Since $u_l\in C_{\rm loc}^{\sigma+\bar\beta}(B_{4\rho^l}(0))$ for each $0\leq l\leq i$, then $U$ is a viscosity solution of 
\begin{equation*}
I^{(i+1)}U=w^{-1}(\rho^{i+1})f(\rho^{i+1}x)-w^{-1}(\rho^{i+1})f(0),\quad \text{in $B_{\frac{4}{\rho}}(0)$},
\end{equation*}
where
\begin{eqnarray*}
I^{(i+1)}U:&=&\inf_{a\in\mathcal{A}}\Big\{\int_{\mathbb R^n}\big(\delta U(x,y)+\sum_{l=0}^{i}\rho^{-(i+1)\sigma}w^{-1}(\rho^{i+1})\delta u_l(\rho^{i+1}x,\rho^{i+1}y)\big)K_a^{i+1}(x,y)dy\Big\}\\
&&-w^{-1}(\rho^{i+1})f(0)\\
&=&\inf_{a\in\mathcal{A}}\Big\{\int_{\mathbb R^n}\big(\delta U(x,y)+\sum_{l=0}^{i}\rho^{-(i+1-l)\sigma}w^{-1}(\rho^{i+1})w(\rho^l)\delta v_l(\rho^{i+1-l}x,\rho^{i+1-l}y)\big)K_a^{i+1}(x,y)dy\Big\}\\
&&-w^{-1}(\rho^{i+1})f(0).
\end{eqnarray*}
It is clear that $I^{(i+1)}$ is uniformly elliptic with respect to $\mathcal{L}_0(\lambda,\Lambda,\sigma)$. Denote 
\begin{eqnarray*}
I_0^{(i+1)}v:&=&\inf_{a\in\mathcal{A}}\Big\{\int_{\mathbb R^n}\big(\delta v(x,y)+\sum_{l=0}^{i}\rho^{-(i+1)\sigma}w^{-1}(\rho^{i+1})\delta u_l(\rho^{i+1}x,\rho^{i+1}y)\big)K_a^{i+1}(0,y)dy\Big\}\\ 
&&-w^{-1}(\rho^{i+1})f(0)\\
&=&\inf_{a\in\mathcal{A}}\Big\{\int_{\mathbb R^n}\big(\delta v(x,y)+\sum_{l=0}^{i}\rho^{-(i+1-l)\sigma}w^{-1}(\rho^{i+1})w(\rho^l)\delta v_l(\rho^{i+1-l}x,\rho^{i+1-l}y)\big)K_a^{i+1}(0,y)dy\Big\}\\
&&-w^{-1}(\rho^{i+1})f(0),
\end{eqnarray*}
which is also uniformly elliptic with respect to $\mathcal{L}_0(\lambda,\Lambda,\sigma)$. Let $v_{i+1}$ be the viscosity solution of 
\begin{equation*}
\left\{\begin{array}{ll}
I_0^{(i+1)}v_{i+1}=0,\quad &\hbox{in $ B_4(0)$},\\
\qquad\,\,\, v_{i+1}=U, &\hbox{in $B_4^c(0)$}.\end{array}\right.
\end{equation*}
With our induction assumption $(\ref{eq4.7})$, it follows that for all $m=0,\cdots,i+1$,
\begin{equation*}
\inf_{a\in\mathcal{A}}\int_{\mathbb R^n}\big(\sum_{l=0}^m\rho^{-(m-l)\sigma}w^{-1}(\rho^m)w(\rho^l)\delta v_l(\rho^{m-l}x,\rho^{m-l}y)\big)K_a^{m}(0,y)dy=w^{-1}(\rho^m)f(0),\quad\text{in $B_4(0)$}.
\end{equation*}
It follows from Theorem \ref{th3.1} that $v_{i+1}\in C_{\rm loc}^{\sigma+\bar\beta}(B_4(0))$ and for any $0<\kappa\leq \sigma+\bar\beta$
\begin{equation*}
\|v_{i+1}\|_{C^{\kappa}(B_{4-\tau}(0))}\leq C_2\tau^{-\kappa}.
\end{equation*}
We then want to prove that 
\begin{equation*}
\|v_{i+1}\|_{L^{\infty}(\mathbb R^n)}\leq\|U\|_{L^{\infty}(\mathbb R^n)}\leq 1.
\end{equation*}
Since $\|u\|_{L^{\infty}(\mathbb R^n)}\leq\frac{1}{2}$, $(\ref{eq4.9})$, $(\ref{eq4.14})$ and $u_l\in C_{\rm loc}^{\sigma+\bar\beta}(B_{4\rho^l}(0))$ for any $0\leq l\leq i$ hold, it follows from Theorem 3.2 in \cite{LL2} that $v_{i+1}\in C(\overline {B_4(0)})$. Suppose that there exists $x_0\in B_4(0)$ such that $v_{i+1}(x_0)=\max_{B_4(0)}v_{i+1}>\|U\|_{L^{\infty}(\mathbb R^n\setminus B_4(0))}$. Then
\begin{equation}\label{eq4.19}
\sup_{a\in\mathcal{A}}\int_{\mathbb R^n}\delta v_{i+1}(x_0,y)K_a^{i+1}(0,y)dy<0.
\end{equation}
Since $I_0^{(i+1)}0(x)=0$ for any $x\in B_4(0)$, then we have 
\begin{equation*}
0= I_0^{(i+1)}v_{i+1}(x)-I_0^{(i+1)}0(x)\leq \sup_{a\in\mathcal{A}}\int_{\mathbb R^n}\delta v_{i+1}(x,y)K_a^{i+1}(0,y)dy,\quad \text{for any $B_4(0)$},
\end{equation*}
which contradicts $(\ref{eq4.19})$. Similarly, we have $v_{i+1}(x)\geq -\|U\|_{L^{\infty}(\mathbb R^n\setminus B_4(0))}$ for any $x\in B_4(0)$. By induction assumptions, we have $\|U\|_{L^{\infty}(\mathbb R^n)}\leq 1$, $U=0$ in $B_{\frac{4}{\rho}}^c(0)$ and 
\begin{equation*}
[U]_{C^{\alpha_1}(B_{\frac{4-3\tau}{\rho}}(0))}\leq 8C_1\frac{w(\rho^i)}{w(\rho^{i+1})}\rho^{\alpha_1-\sigma}\tau^{-3}\leq 8C_1\rho^{-3}\tau^{-3}.
\end{equation*}
By Lemma \ref{le2.3}, we have, for any $x_1, x_2\in B_4(0)$
\begin{eqnarray*}
&&\Big|\int_{\mathbb R^n}\sum_{l=0}^{i}\rho^{-(i+1-l)\sigma}w^{-1}(\rho^{i+1})w(\rho^l)\big(\delta v_l(\rho^{i+1-l}x_1,\rho^{i+1-l}y)-\delta v_l(\rho^{i+1-l}x_2,\rho^{i+1-l}y)\big)K_a^{i+1}(0,y)dy\Big|\\
&=&\Big|\sum_{l=0}^i\rho^{-(i+1-l)\sigma}w^{-1}(\rho^{i+1})w(\rho^l)\int_{\mathbb R^n}\big(\delta v_l(\rho^{i+1-l}x_1,\rho^{i+1-l}y)-\delta v_l(\rho^{i+1-l}x_2,\rho^{i+1-l}y)\big)K_a^{i+1}(0,y)dy\Big|\\
&=&\Big|\sum_{l=0}^{i}w^{-1}(\rho^{i+1})w(\rho^l)\int_{\mathbb R^n}\big(\delta v_l(\rho^{i+1-l}x_1,y)-\delta v_l(\rho^{i+1-l}x_2,y)\big)K_a^{l}(0,y)dy\Big|\\
&\leq&\sum_{l=0}^iw^{-1}(\rho^{i+1})w(\rho^l)\|L_a^{l}v_l\|_{C^{\bar\beta}(B_{4\rho^{i+1-l}}(0))}\rho^{(i+1-l)\bar\beta}|x_1-x_2|^{\bar\beta}\\
&\leq&\sum_{l=0}^{i}w^{-1}(\rho^{i+1})w(\rho^l)C(\|v_l\|_{C^{\sigma+\bar\beta}(B_{5\rho^{i+1-l}}(0))}+\|v_l\|_{L^{\infty}(\mathbb R^n)})\rho^{(i+1-l)\bar\beta}|x_1-x_2|^{\bar\beta}\\
&\leq&\sum_{l=0}^iw^{-1}(\rho^{i+1})w(\rho^l)C(C_2+1)\rho^{(i+1-l)\bar\beta}|x_1-x_2|^{\bar\beta}\\
&\leq&\sum_{l=0}^iC\frac{w(\rho^l)}{w(\rho^{i+1})}\rho^{(i+1-l)\bar\beta}|x_1-x_2|^{\bar\beta}\leq C\rho^{\bar\beta-\beta}|x_1-x_2|^{\bar\beta}.
\end{eqnarray*}
Then we will show that we can choose $\gamma$ sufficiently small such that
\begin{equation}\label{a4.19}
\|I^{(i+1)}-I_0^{(i+1)}\|_{B_4(0)}\leq \eta_2\leq 1,
\end{equation}
where $\eta_2$ is given in Lemma \ref{le2.6} depending on $\epsilon=\rho^{\sigma+\bar\beta}$, $R_0=\frac{4}{\rho}$, $M_0=C\rho^{\bar\beta-\beta}$, $M_1=1$, $M_2=8C_1\rho^{-3}$ and $M_3=C_2$. For any $x\in B_4(0)$, $h\in C^{1,1}(x)$, $\|h\|_{L^{\infty}(\mathbb R^n)}\leq M$, $|h(y)-h(x)-(y-x)\cdot\nabla h(x)|\leq \frac{M}{2}|x-y|^2$ for any $y\in B_1(x)$, we have
\begin{eqnarray*}
\|I^{(i+1)}&-&I_0^{(i+1)}\|_{B_4(0)}\leq \frac{1}{M+1}\sup_{a\in\mathcal{A}}\Big|\int_{\mathbb R^n}\delta h(x,y)\big(K_a^{i+1}(x,y)-K_a^{i+1}(0,y)\big)dy\Big|\\
&&+\sum_{l=0}^i\sup_{a\in\mathcal{A}}\Big|\int_{\mathbb R^n}\rho^{-(i+1)\sigma}w^{-1}(\rho^{i+1})\delta u_l(\rho^{i+1}x,\rho^{i+1}y)\big(K_a^{i+1}(x,y)-K_a^{i+1}(0,y)\big)dy\Big|\\
&&=I_1+I_2.
\end{eqnarray*}
It follows from the same computation as that in $(\ref{eq4.6})$ that 
\begin{equation*}
|I_1|\leq 5\gamma w(5).
\end{equation*}
By $(\ref{eq4.11})$, we have, for any $a\in\mathcal{A}$, $l=0,\cdots,i$ and $x\in B_4(0)$
\begin{eqnarray*}
&&\Big|\int_{\mathbb R^n}\delta u_l(\rho^{i+1}x,\rho^{i+1}y)\big(K_a^{i+1}(x,y)-K_a^{i+1}(0,y)\big)dy\Big|\\
&\leq&\rho^{\sigma(i+1)}\int_{\mathbb R^n}|\delta u_l(\rho^{i+1}x,y)||K_a(\rho^{i+1}x,y)-K_a(0,y)|dy\\
&\leq&\rho^{\sigma(i+1)}\int_{B_{\rho^l}(0)}C_2\rho^{-\min\{2-\sigma,\bar\beta \}l}w(\rho^l)|y|^{\min\{2,\sigma+\bar\beta\}}|K_a(\rho^{i+1}x,y)-K_a(0,y)|dy\\
&&+\rho^{\sigma(i+1)}\int_{\mathbb R^n\setminus B_{\rho^l}(0)}4\rho^{l\sigma}w(\rho^l)|K_a(\rho^{i+1}x,y)-K_a(0,y)|dy\\
&\leq&(C_2+4)\rho^{\sigma(i+1)}w(\rho^l)\gamma w(\rho^{i+1}|x|)\\
&\leq&(C_2+4)\rho^{\sigma(i+1)}w(\rho^l)\gamma w(\rho^i).
\end{eqnarray*}
Thus, we have
\begin{equation*}
I_2\leq (C_2+4)\frac{w(\rho^i)}{w(\rho^{i+1})}\gamma\sum_{l=0}^iw(\rho^l)\leq (C_2+4)\rho^{-1}\gamma\sum_{l=0}^{+\infty}w(\rho^l)<+\infty.
\end{equation*}
We finally choose $\gamma$ such that 
\begin{equation*}
\gamma\leq\min\Big\{\frac{\eta_2}{\big(5w(5)+(C_2+4)\rho^{-1}\sum_{l=0}^{+\infty}w(\rho^l)\big)},\frac{\eta_1}{6w(5)},\frac{1}{(C_2+4)w(4)}\Big\}.
\end{equation*}
Therefore, $(\ref{a4.19})$ holds.
By Lemma \ref{le2.6}, we have
\begin{equation*}
\|U-v_{i+1}\|_{L^{\infty}(\mathbb R^n)}=\|U-v_{i+1}\|_{L^{\infty}(B_4(0))}\leq\epsilon=\rho^{\sigma+\bar\beta}\leq\rho^{\sigma}\frac{w(\rho^{i+2})}{w(\rho^{i+1})}.
\end{equation*}
Let 
\begin{equation*}
u_{i+1}(x)=\rho^{\sigma(i+1)}w(\rho^{i+1})v_{i+1}(\rho^{-(i+1)}x),
\end{equation*}
and 
\begin{equation*}
V=U-v_{i+1}=\rho^{-\sigma(i+1)}w^{-1}(\rho^{i+1})\big(u-\sum_{l=0}^{i+1}u_l\big)(\rho^{i+1}x).
\end{equation*}
Then, for any $x\in B_4(0)$ we have
\begin{eqnarray*}
\bar I^{(i+1)}V:&=&\inf_{a\in\mathcal{A}}\int_{\mathbb R^n}\delta V(x,y)+\sum_{l=0}^{i+1}\rho^{-\sigma(i+1)}w^{-1}(\rho^{i+1})\delta u_l(\rho^{i+1}x,\rho^{i+1}y)K_a^{i+1}(x,y)dy\\
&&-w^{-1}(\rho^{i+1})f(0)\\
&=&w^{-1}(\rho^{i+1})f(\rho^{i+1}x)-w^{-1}(\rho^{i+1})f(0).
\end{eqnarray*}
Moreover, we have for any $x\in B_{4-2\tau}(0)$
\begin{eqnarray*}
\bar I^{(i+1)}0&=&\inf_{a\in\mathcal{A}}\Big\{\int_{\mathbb R^n}\sum_{l=0}^{i+1}\rho^{-\sigma(i+1)}w^{-1}(\rho^{i+1})\delta u_l(\rho^{i+1}x,\rho^{i+1}y)K_a^{i+1}(x,y)dy\Big\}\\
&&-w^{-1}(\rho^{i+1})f(0)\\
&=&\inf_{a\in\mathcal{A}}\Big\{\int_{\mathbb R^n}\sum_{l=0}^{i+1}\rho^{-\sigma(i+1)}w^{-1}(\rho^{i+1})\delta u_l(\rho^{i+1}x,\rho^{i+1}y)K_a^{i+1}(x,y)dy\Big\}\\
&&-\inf_{a\in\mathcal{A}}\Big\{\int_{\mathbb R^n}\sum_{l=0}^{i+1}\rho^{-(i+1)\sigma}w^{-1}(\rho^{i+1})\delta u_l(\rho^{i+1}x,\rho^{i+1}y)K_a^{i+1}(0,y)dy\Big\}\\
&\leq&\sup_{a\in\mathcal{A}}\Big\{\sum_{l=0}^{i+1}\int_{\mathbb R^n}\rho^{-\sigma(i+1)}w^{-1}(\rho^{i+1})\delta u_l(\rho^{i+1}x,\rho^{i+1}y)\big(K_a^{i+1}(x,y)-K_a^{i+1}(0,y)\big)dy\Big\}\\
&\leq&(C_2+4)\rho^{-1}\tau^{-\sigma}\gamma\sum_{l=0}^{+\infty}w(\rho^l)\leq\eta_2\tau^{-\sigma}
\leq\tau^{-\sigma}.
\end{eqnarray*}
It is clear that $\bar I^{(i+1)}$ is uniformly elliptic with respect to $\mathcal{L}_0(\lambda,\Lambda,\sigma)$. Thus, for any $x\in B_{4-2\tau}(0)$
\begin{eqnarray*}
M_{\mathcal{L}_0}^+V&\geq& \bar I^{(0)}V-\bar I^{(0)}0=w^{-1}(\rho^{i+1})(f(\rho^{i+1}x)-f(0))-\tau^{-\sigma}\\
&\geq&-\gamma w^{-1}(\rho^{i+1})w(\rho^{i+1}|x|)-\tau^{-\sigma}\\
&\geq&-\gamma w^{-1}(\rho^{i+1})w(4\rho^{i+1})-\tau^{-\sigma}\\
&\geq&-\gamma\rho^{-1}-\tau^{-\sigma},
\end{eqnarray*}
and similarly,
\begin{equation*}
M_{\mathcal{L}_0}^-V\leq\gamma\rho^{-1}+\tau^{-\sigma}.
\end{equation*}
It follows from Theorem 12.1 of \cite{LL1} that,
\begin{eqnarray*}
[V]_{C^{\alpha_1}(B_{4-3\tau}(0))}&\leq& C_1\tau^{-\alpha_1}(\|V\|_{L^{\infty}(\mathbb R^n)}+\gamma\rho^{-1}+\tau^{-\sigma})\\
&\leq&C_1\tau^{-\alpha_1}(\epsilon+\gamma\rho^{-1}+\tau^{-\sigma})\\
&\leq& 8C_1\tau^{-3}.
\end{eqnarray*}
Thus, we finish the proof.
\end{proof}

\begin{corollary}
Assume that $2>\sigma\geq \sigma_0>0$ and $K_a(x,y)\in \mathcal{L}_2(\lambda,\Lambda,\sigma)$ for any $a\in\mathcal{A}$. Assume that $w(t)$ is a Dini modulus of continuity satisfying $(H2)_{\bar\beta,\sigma}$, where $\bar\beta$ is given in Theorem \ref{th3.1}. Assume that there exists $C_f>0$ such that, for any $x_1,x_2\in B_1(0)$
\begin{equation*}
|f(x_1)-f(x_2)|\leq C_f w(|x_1-x_2|)\,\, \text{and}\,\, \|f\|_{L^{\infty}(B_1(0))}\leq C_f
\end{equation*}
and $K_a(x,y)$ satisfies, for any $0<r\leq 1$
\begin{equation*}
\int_{\mathbb R^n}|K_a(x_1,y)-K_a(x_2,y)|\min\{|y|^{\min\{2,\sigma+\bar\beta\}},r^{\min\{2,\sigma+\bar\beta\}}\}dy\leq \Lambda w(|x_1-x_2|)r^{\min\{2-\sigma,\bar\beta\}}.
\end{equation*}
If $u$ is a bounded viscosity solution of $(\ref{eq1.1})$, then there exists a constant $C>0$ depending on $\lambda,\Lambda,n,\sigma_0$, $\sigma$ and $w$ such that
\begin{equation*}
\|u\|_{C^{\sigma}(B_\frac{1}{2}(0))}\leq C(\|u\|_{L^{\infty}(\mathbb R^n)}+C_f).
\end{equation*}
\end{corollary}

\begin{example}
Since the assumption $(\ref{hs})$ is slightly complicated, we provide several examples when it is satisfied. We first consider the kernel $K_a(x,y)$ which satisfies, for any $r>0$
\begin{equation}\label{hs1}
\int_{B_{2r}(0)\setminus B_{r}(0)}|K_a(x,y)-K_a(0,y)|dy\leq \Lambda w(|x|)r^{-\sigma},\quad \text{in $B_1(0)$}.
\end{equation}
Thus, for any $0<r<1$, $x\in B_1(0)$ and non-negative integer $n$, we have 
\begin{equation*}
\int_{B_{\frac{r}{2^{n}}}(0)\setminus B_{\frac{r}{2^{n+1}}}(0)}|K_a(x,y)-K_a(0,y)||y|^{\min\{2,\sigma+\bar\beta\}}dy\leq \Lambda w(|x|)2^{\sigma-n\min\{2-\sigma,\bar\beta\}}r^{\min\{2-\sigma,\bar\beta\}},
\end{equation*}
and
\begin{equation*}
\int_{B_{2^{n+1}r}(0)\setminus B_{2^{n}r}(0)}|K_a(x,y)-K_a(0,y)||r|^{\min\{2,\sigma+\bar\beta\}}dy\leq \Lambda w(|x|)2^{-n\sigma}r^{\min\{2-\sigma,\bar\beta\}}.
\end{equation*}
Then it is not hard to verify that $(\ref{hs1})$ implies $(\ref{hs})$. Another more concrete example satisfying $(\ref{hs})$ is given by the kernel of the form 
\begin{equation}
K_a(x,y)=\frac{k_a(x,y)}{|y|^{n+\sigma}},\quad \text{for any $x\in B_1(0)$ and $y\in\mathbb R^n$},
\end{equation}
where $|k_a(x,y)-k_a(0,y)|\leq \Lambda w(|x|)$. 
\end{example}

\textbf{Acknowledgement.} I would like to thank my advisor Prof. Andrzej \Swiech\, for all the useful discussions and encouragement.

\end{document}